%% file: ALMforNLP.tex
\documentclass[smallextended,final]{svjour3}
\usepackage[letterpaper,top=2in, bottom=1.5in, left=1in, right=1in]{geometry}

\usepackage{epsfig,epsf,fancybox}
\usepackage{amsmath}
\usepackage{mathrsfs}
\usepackage{amssymb}
\usepackage{amsfonts}
\usepackage{graphicx}
\usepackage{color}
\usepackage[linktocpage,pagebackref,colorlinks,linkcolor=blue,anchorcolor=blue,citecolor=blue,urlcolor=blue,hypertexnames=false]{hyperref}
\usepackage{boxedminipage}
\usepackage{stmaryrd}
\usepackage{multirow}
\usepackage{booktabs}
\usepackage{cite}
\usepackage{float}
\usepackage[ruled,vlined,linesnumbered]{algorithm2e}

\input{macros.tex}
\begin{document}

\title{Iteration complexity of inexact augmented Lagrangian methods for constrained convex programming}


\author{Yangyang Xu}

\institute{Y. Xu \at Department of Mathematical Sciences, Rensselaer Polytechnic Institute, Troy, NY 12180\\
\email{xuy21@rpi.edu}}

\date{}

\maketitle

\begin{abstract}
Augmented Lagrangian method (ALM) has been popularly used for solving constrained optimization problems. Practically, subproblems for updating primal variables in the framework of ALM usually can only be solved inexactly. The convergence and local convergence speed of ALM have been extensively studied.  However, the global convergence rate of inexact ALM is still open for problems with nonlinear inequality constraints. In this paper, we work on general convex programs with both equality and inequality constraints. For these problems, we establish the global convergence rate of inexact ALM and estimate its iteration complexity in terms of the number of gradient evaluations to produce a solution with a specified accuracy.

We first establish an ergodic convergence rate result of inexact ALM that uses constant penalty parameters or geometrically increasing penalty parameters. Based on the convergence rate result, we apply Nesterov's optimal first-order method on each primal subproblem and estimate the iteration complexity of the inexact ALM. We show that if the objective is convex, then $O(\vareps^{-1})$ gradient evaluations are sufficient to guarantee an $\vareps$-optimal solution in terms of both primal objective and feasibility violation. If the objective is strongly convex, the result can be improved to $O(\vareps^{-\frac{1}{2}}|\log\vareps|)$. Finally, by relating to the inexact proximal point algorithm, we establish a nonergodic convergence rate result of inexact ALM that uses geometrically increasing penalty parameters. We show that the nonergodic iteration complexity result is in the same order as that for the ergodic result. Numerical experiments on quadratically constrained quadratic programming are conducted to compare the performance of the inexact ALM with different settings.

\vspace{0.3cm}

\noindent {\bf Keywords:} augmented Lagrangian method (ALM), nonlinearly constrained problem, first-order method, global convergence rate, iteration complexity
\vspace{0.3cm}

\noindent {\bf Mathematics Subject Classification:} 90C06, 90C25, 68W40, 49M27.

\end{abstract}

\section{Introduction}
In this paper, we consider the constrained convex programming
\begin{equation}\label{eq:ccp}
\Min_{\vx\in\cX} f_0(\vx), \st \vA\vx=\vb, f_i(\vx)\le 0, i=1,\ldots, m,
\end{equation}
where $\cX\subseteq\RR^n$ is a closed convex set, $\vA$ and $\vb$ are respectively given matrix and vector, and $f_i$ is a convex function for every $i=0,1,\ldots,m$. Any convex optimization problem can be written in the standard form of \eqref{eq:ccp}. It appears in many areas including statistics, machine learning, data mining, engineering, signal processing, finance, operations research, and so on. 

Note that the constraint $\vx\in \cX$ can be equivalently represented by using an inequality constraint $\iota_\cX(\vx) \le 0$ or adding $\iota_\cX(\vx)$ to the objective, where $\iota_\cX$ denotes the indicator function on $\cX$. However, we explicitly use it for technical reason. In addition, 
every affine constraint $\va_j^\top \vx=b_j$ can be equivalently represented by two inequality constraints: $\va_j^\top \vx - b_j \le 0$ and $-\va_j^\top \vx + b_j \le 0$. That way does not change theoretical results of an algorithm but will make the problem computationally more difficult. 


One popular method for solving \eqref{eq:ccp} is the augmented Lagrangian method (ALM), which first appeared in \cite{hestenes1969multiplier, powell1969method}. ALM alternatingly updates the primal variable and the Lagrangian multipliers. At each update, the primal variable is renewed by minimizing the augmented Lagrangian (AL) function and the multipliers by a dual gradient ascent. The global convergence and local convergence rate of ALM have been extensively studied; see the books \cite{bertsekas1999nonlinear, bertsekas2014constrained}. Several recent works (e.g., \cite{he2010aalm, he2012-rate-drs}) establish the global convergence rate of ALM and/or its variants for affinely constrained problems. In the framework of ALM, the primal subproblem usually can only be solved inexactly, and thus practically inexact ALM (iALM) is often used. However, to the best of our knowledge, the global convergence rate of iALM for problems with nonlinear inequality constraints still remains open\footnote{Although the global convergence rate in terms of augmented dual objective can be easily shown from existing works (e.g., see our discussion in section \ref{sec:review}), that does not indicate the convergence speed from the perspective of the primal objective and feasibility.}. We address this open question in this work and also establish the iteration complexity of iALM in terms of the number of gradient evaluations. The iteration complexity result appears to be optimal.

We will assume composite convex structure on \eqref{eq:ccp}. More specifically, we assume \begin{equation}\label{eq:f0}
f_0(\vx)=g(\vx)+h(\vx),
\end{equation} 
where $g$ is a Lipschitz differentiable convex function, and $h$ is a simple\footnote{By ``simple'', we mean the proximal mapping of $h$ is easy to evaluate, i.e., it is easy to find a solution to $\min_{\vx\in\cX} h(\vx) + \frac{1}{2\gamma}\|\vx-\hat{\vx}\|^2$ for any $\hat{\vx}$ and $\gamma>0$.} (possibly nondifferentiable) convex function. Also, $f_i$ is convex and Lipschitz differentiable for every $i\in [m]$, namely, there are constants $L_0,L_1,\ldots, L_m$ such that
\begin{subequations}\label{eq:lip-gf}
\begin{align}
\|\nabla g(\hat{\vx})-\nabla g(\tilde{\vx})\|\le L_0\|\hat{\vx}-\tilde{\vx}\|,\,&\forall\, \hat{\vx}, \tilde{\vx} \in \dom(h)\cap \cX,\label{eq:lip-g}\\
\|\nabla f_i(\hat{\vx})-\nabla f_i(\tilde{\vx})\|\le L_i\|\hat{\vx}-\tilde{\vx}\|,\,& \forall\, \hat{\vx}, \tilde{\vx} \in \dom(h) \cap \cX, \forall i\in [m].\label{eq:lip-f}
\end{align} 
\end{subequations}
In addition, we assume the boundedness of $\dom(h)\cap\cX$ and denote its diameter as 
$$D=\Max_{\hat{\vx},\tilde{\vx}\in \dom(h)\cap\cX}\|\hat{\vx}-\tilde{\vx}\|.$$

\subsection{Augmented Lagrangian function}
In the literature, there are several different penalty terms used in an augmented Lagrangian (AL) function, such as the classic one \cite{rockafellar1973dual, rockafellar1973multiplier}, the quadratic penalty on constraint violation \cite{bertsekas1973convergence}, and the exponential penalty \cite{tseng1993convergence-exp-mult}. The work \cite{ben1997penalty-barrier} gives a general class of augmented penalty functions that satisfy certain properties. In this paper, we use the classic one. As discussed below, it can be derived from a quadratic penalty on an equivalent equality constrained problem.  

Introducing nonnegative slack variable $s_i$'s, one can write \eqref{eq:ccp} to an equivalent form:
\begin{equation}\label{eq:eq-ccp}
\Min_{\vx\in\cX, \vs\ge \vzero} f_0(\vx), \st \vA\vx=\vb, f_i(\vx) + s_i = 0, i=1,\ldots, m.
\end{equation}
With quadratic penalty on the equality constraints, the AL function of \eqref{eq:eq-ccp} is
$$\tilde{\cL}_\beta(\vx,\vs,\vy,\vz)= f_0(\vx) + \vy^\top(\vA\vx-\vb) + \sum_{i=1}^m z_i\big(f_i(\vx) + s_i\big)+\frac{\beta}{2}\|\vA\vx-\vb\|^2+\frac{\beta}{2}\sum_{i=1}^m\big(f_i(\vx) + s_i\big)^2,$$
where $\vy$ and $\vz$ are multipliers, and $\beta$ is the augmented penalty parameter.
Minimizing $\tilde{\cL}_\beta$ with respect to $\vs\ge\vzero$ while fixing $\vx,\vy$ and $\vz$, we have the optimal $\vs$ given by
$$s_i=\left[-\frac{z_i}{\beta}-f_i(\vx)\right]_+,\,i=1,\ldots,m.$$
Plugging the above $\vs$ into $\tilde{\cL}_\beta$ gives
$$
\tilde{\cL}_\beta(\vx,\vs,\vy,\vz)= f_0(\vx) + \vy^\top(\vA\vx-\vb) + \frac{\beta}{2}\|\vA\vx-\vb\|^2+\sum_{i=1}^m\psi_\beta(f_i(\vx),z_i),
$$
where
\begin{equation}\label{eq:def-psi}
\psi_\beta(u,v) = \left\{\begin{array}{ll}uv+\frac{\beta}{2}u^2,\,&\text{ if }\beta u +v \ge0,\\[0.2cm]
-\frac{v^2}{2\beta},\,&\text{ if }\beta u +v < 0.\end{array}\right.
\end{equation}
Let
$$\Psi_\beta(\vx,\vz)=\sum_{i=1}^m\psi_\beta(f_i(\vx),z_i),$$
and we obtain the classical AL function of \eqref{eq:ccp}:
\begin{equation}\label{eq:aug-fun}
\cL_\beta(\vx,\vy,\vz) = f_0(\vx) + \vy^\top(\vA\vx-\vb) + \frac{\beta}{2}\|\vA\vx-\vb\|^2 + \Psi_\beta(\vx,\vz).
\end{equation}

\subsection{Inexact augmented Lagrangian method}
The augmented Lagrangian method (ALM) was proposed in \cite{hestenes1969multiplier, powell1969method}. Within each iteration, ALM first updates the $\vx$ variable by minimizing the AL function with respect to $\vx$ while fixing $\vy$ and $\vz$, and then it performs a dual gradient ascent update to $\vy$ and $\vz$. In general, it is difficult to exactly minimize the AL function about $\vx$. A more realistic way is to solve the $\vx$-subproblem within a tolerance error, which leads to the inexact ALM. Its pseudocode is given in Algorithm \ref{alg:ialm} below. If $\vareps_k=0,\,\forall k$, it reduces to the ALM. 
\begin{algorithm}[h]
\caption{Inexact augmented Lagrangian method for \eqref{eq:ccp}}\label{alg:ialm}
\DontPrintSemicolon
\textbf{Initialization:} choose $\vx^0, \vy^0, \vz^0$ and $\{\beta_k,\rho_k\}$\;
\For{$k=0,1,\ldots$}{
Find $\vx^{k+1}\in\cX$ such that 
\begin{equation}\label{eq:ialm-x}
\cL_{\beta_k}(\vx^{k+1}, \vy^k , \vz^k  ) \le \min_{\vx\in\cX} \cL_{\beta_k}(\vx, \vy^k , \vz^k  ) + \vareps_k.
\end{equation}

Update $\vy$ and $\vz$ by
\begin{align}
\vy^{k+1} =&~\vy^k  + \rho_k (\vA\vx^{k+1}-\vb),\label{eq:alm-y}\\
z^{k+1} _i=&~z_i^k+\rho_k \cdot\max\left(-\frac{z_i^k}{\beta_k}, f_i(\vx^{k+1})\right), i=1,\ldots, m.\label{eq:alm-z}
\end{align}
}
\end{algorithm}

It is shown in \cite{rockafellar1973dual} that the augmented dual function\footnote{Although \cite{rockafellar1973dual} only considers the inequality constrained case, the results derived there apply to the case with both equality and inequality constraints.}
$d_\beta(\vy,\vz)=\min_{\vx\in \cX}\cL_\beta(\vx,\vy,\vz)$
is continuously differentiable, and $\nabla d_\beta$ is Lipschitz continuous with constant $\frac{1}{\beta}$. In addition, it turns out that the (inexact) ALM is an (inexact) augmented dual gradient ascent \cite{rockafellar1973multiplier}, and thus convergence rate of the (inexact) ALM in term of $d_\beta$ can be shown from existing results about (inexact) gradient method \cite{schmidt2011convergence-ipg}. However, it is unclear from there how to get convergence rate result from the perspective of the primal problem, especially as $\beta$ varies at different updates. Our analysis will be different from this line, and our results will be based on the primal problem. In addition, our requirement on the tolerance errors is weaker than that in \cite{rockafellar1973multiplier}. 

The main results we establish in this paper are summarized as follows. Both ergodic and nonergodic convergence rate results are established.
\begin{theorem}[Summary of main results]
For a given $\vareps>0$, choose a positive integer $K$ and numbers $C_\beta>0, C_\vareps>0$. Let $\{(\vx^k,\vy^k,\vz^k)\}_{k=0}^{K}$ be the iterates generated from Algorithm \ref{alg:ialm} with parameters set according to one of the follows: 
\begin{enumerate}
\item $\rho_k=\beta_k=\frac{C_\beta}{K\vareps}, \vareps_k=\frac{\vareps}{2}\frac{C_\vareps}{C_\beta},\,\forall k$. 
\item $\rho_k=\beta_k=\beta_g\sigma^k, \forall k$ for certain $\beta_g>0$ and $\sigma>1$ such that $\sum_{k=0}^{K-1}\beta_k=\frac{C_\beta}{\vareps}$, and $\vareps_k=\frac{\vareps}{2}\frac{C_\vareps}{C_\beta},\,\forall k$.

\item $\rho_k=\beta_k=\beta_g\sigma^k, \forall k$ for certain $\beta_g>0$ and $\sigma>1$ such that $\sum_{k=0}^{K-1}\beta_k=\frac{C_\beta}{\vareps}$. If $f_0$ is convex, let $\vareps_k=\frac{C_\vareps}{2\beta_k^{\frac{1}{3}}}\frac{1}{\sum_{t=0}^{K-1}\beta_t^{\frac{2}{3}}},\,\forall k$, and if $f_0$ is strongly convex, let $\vareps_k=\frac{C_\vareps}{2\beta_k^{\frac{1}{2}}}\frac{1}{\sum_{t=0}^{K-1}\beta_t^{\frac{1}{2}}},\,\forall k$

\end{enumerate}
Then the averaged point $\bar{\vx}^K=\sum_{k=0}^{K-1}\frac{\rho_k\vx^{k+1}}{\sum_{t=0}^{K-1}\rho_t}$ is an $O(\vareps)$-optimal solution (see Definition \ref{def:eps-opt}), where the hidden constant depends on $C_\beta, C_\vareps$ and dual solution $(\vy^*,\vz^*)$, and for the second and third settings, the actual point $\vx^K$ is also an $O(\vareps)$-optimal solution. In addition, the total number of evaluations on $\nabla g$ and $\nabla f_i, i=1,\ldots,m$ is $O(\vareps^{-1})$ if $f_0$ is convex and $O(\vareps^{-\frac{1}{2}}|\log\vareps|)$ if $f_0$ is strongly convex.
\end{theorem}

The formal statement and the hidden constants are shown in Theorem \ref{thm:it-comp-const} for the first setting, in Theorems \ref{thm:iter-comp-geo-const-eps} and \ref{thm:nonerg-iter} for the second setting, and in Theorems \ref{thm:iter-comp-geo} and \ref{thm:nonerg-iter} for the third setting.

\subsection{Contributions}


The contributions of this paper are mainly on establishing global convergence rate and estimating iteration complexity of iALM for solving \eqref{eq:ccp}, in both ergodic and nonergodic sense. They are listed as follows. 
\begin{itemize}\renewcommand\labelitemi{--}
\item We first establish an ergodic convergence rate result of iALM through a novel analysis. With penalty parameters fixed to constant or increased geometrically, we choose the tolerance errors accordingly. By applying Nesterov's optimal first-order method to each $\vx$-subproblem, we show that to reach an $\vareps$-optimal solution, $O(\vareps^{-1})$ gradient evaluations are sufficient if the objective is convex, and the order is improved to $O(\vareps^{-\frac{1}{2}}|\log\vareps|)$ if the objective is strongly convex. For the convex case, the result is optimal, and for the strongly convex, the result also appears the best compared to existing ones in the literature; see the discussions in Remark \ref{rm:iter-to-rate}. 

\item We note that if iALM only runs to one iteration, i.e., a single $\vx$-subproblem \eqref{eq:ialm-x} is solved, then the algorithm reduces to the penalty method by setting initial multipliers to zero vectors. Hence, as a byproduct, we establish the iteration complexity result of the penalty method for solving \eqref{eq:ccp}. The work \cite{lan2013iteration-penalty} analyzes first-order penalty method for solving affinely constrained problems. To have an $\vareps$-optimal solution, $O(\vareps^{-2})$ gradient evaluations are required, if the penalty method is applied to the original problem, and $O(\vareps^{-1}|\log\vareps|)$ gradient evaluations are required if the penalty method is applied to a carefully perturbed problem. Hence, our result is better than the former case by an order $O(\vareps^{-1})$ and the latter case by $O(|\log\vareps|)$.

\item By relating to inexact proximal point algorithm (iPPA), we then establish a nonergodic convergence rate result of iALM through applying existing results in \cite{rockafellar1976augmented} about iALM. We show that with geometrically increasing penalty parameters, the nonergodic convergence of iALM enjoys the same order as that of the ergodic convergence, and the constant is just a few times larger. Compared to one recent nonergodic convergence result \cite{liu2016iALM} for solving affinely constrained problem, our result is better by an order $O(\vareps^{-3})$.
 
\end{itemize}

\subsection{Notation}
For simplicity, throughout the paper, we focus on  a finite-dimensional Euclidean space, but our analysis can be directly extended to a general Hilbert space. 

We use italic letters $a, c, B, L,\ldots,$ for scalars, bold lower-case letters $\vx,\vy,\vz,\ldots$ for vectors, and bold upper-case letters $\vA,\vB,\ldots$ for matrices. $z_i$ denotes the $i$-th entry of a vector $\vz$. We use $\vzero$ to denote a vector of all zeros, and its size is clear from the context. $[m]$ denotes the set $\{1,2,\ldots,m\}$ for any positive integer $m$. Given a real number $a$, we let $[a]_+=\max(0,a)$ and $\lceil a\rceil$ be the smallest integer that is no less than $a$. For a vector $\va$, $[\va]_+$ takes the positive part of $\va$ in a component-wise manner. $\|\va\|$ denotes the Euclidean norm of a vector $\va$ and $\|\vA\|$ the spectral norm of a matrix $\vA$. 

We denote $\vell$ as the vector consisting of $L_i, i\in [m]$, where $L_i$ is the Lipschitz constant of $\nabla f_i$ in \eqref{eq:lip-f}. Also we let $\vf$ be the vector function with $f_i$ as the $i$-th component scalar function. That is 
\begin{equation}\label{eq:vec-ell}
\vell=[L_1,\ldots,L_m], \quad \vf(\vx)=[f_1(\vx),\ldots, f_m(\vx)].
\end{equation}
Given a convex function $f$, $\tilde{\nabla}f(\vx)$ represents one subgradient of $f$ at $\vx$, namely,
$$f(\hat{\vx})\ge f(\vx) + \langle \tilde{\nabla} f(\vx), \hat{\vx}-\vx\rangle,\,\forall\, \hat{\vx},$$ and $\partial f(\vx)$ denotes its subdifferential, i.e., the set of all subgradients. When $f$ is differentiable, we simply write its subgradient as $\nabla f(\vx)$. For a convex set $\cX$, we use $\iota_\cX$ as its indicator function, i.e., 
$$\iota_\cX(\vx)=\left\{\begin{array}{ll}0,&\text{ if }\vx\in\cX,\\+\infty,&\text{ if }\vx\not\in\cX,\end{array}\right.$$
and $\cN_\cX(\vx)=\partial \iota_\cX(\vx)$ as its normal cone at $\vx\in \cX$. 

Given an $\vareps>0$, the $\vareps$-optimal solution of \eqref{eq:ccp} is defined as follows.
\begin{definition}[$\vareps$-optimal solution]\label{def:eps-opt}
Let $f_0^*$ be the optimal value of \eqref{eq:ccp}. Given $\vareps\ge 0$, a point $\vx\in\cX$ is called an $\vareps$-optimal solution to \eqref{eq:ccp} if 
$$|f_0(\vx)-f_0^*| \le \vareps,\text{ and }\|\vA\vx-\vb\|+\big\| [\vf(\vx)]_+\big\| \le \vareps.$$
\end{definition}


\subsection{Outline} The rest of the paper is organized as follows. In section \ref{sec:prepare}, we give a few preparatory results and review Nesterov's optimal first-order method for solving a composite convex program. An ergodic convergence rate result of iALM is given in section \ref{sec:ergo-rate}, and a nonergodic convergence rate result is shown in section \ref{sec:nonergo-rate}. Iteration complexity results in terms of the number of gradient evaluations are established for both ergodic and nonergodic cases. 
Related works are reviewed and compared in section \ref{sec:review}, and numerical results are provided in section \ref{sec:numerical}. Finally section \ref{sec:conclusion} concludes the paper.

\section{Preliminary results and Nesterov's optimal first-order method}\label{sec:prepare}
In this section, we give a few preliminary results and also review Nesterov's optimal first-order method for composite convex programs.

\subsection{Basic facts}
A point $(\vx,\vy,\vz)$ satisfies the Karush-Kuhn-Tucker (KKT) conditions for \eqref{eq:ccp} if
\begin{subequations}\label{eq:kkt}
\begin{align}
0\in \partial f_0(\vx) +\cN_\cX(\vx) + \vA^\top \vy + \sum_{i=1}^m z_i \nabla f_i(\vx), \label{eq:kkt-1st}&\\
\vA\vx = \vb, \quad \vx\in \cX, \label{eq:kkt-res}&\\
z_i\ge0, \quad f_i(\vx)\le 0,\quad z_if_i(\vx)=0, \forall i\in [m]. \label{eq:kkt-cp}
\end{align}
\end{subequations}
From the convexity of $f_i$'s, if $(\vx^*,\vy^*,\vz^*)$ is a KKT point, then
\begin{equation}\label{eq:opt}
f_0(\vx)-f_0(\vx^*)+\langle \vy^*, \vA\vx-\vb\rangle + \sum_{i=1}^m z^*_i f_i(\vx) \ge 0,\, \forall\, \vx\in\cX.
\end{equation}


The result below will be used to establish convergence rate of Algorithm \ref{alg:ialm}. 
\begin{lemma}\label{lem:pre-rate}
Assume $(\vx^*,\vy^*,\vz^*)$ satisfies the KKT conditions in \eqref{eq:kkt}. Let $\bar{\vx}$ be a point such that for any $\vy$ and any $\vz\ge\vzero$,
\begin{equation}\label{eq:pre-rate-cond}
f_0(\bar{\vx})-f_0(\vx^*)+ \vy^\top (\vA\bar{\vx}-\vb) +\sum_{i=1}^m z_i f_i(\bar{\vx}) \le \alpha + c_1\|\vy\|^2+c_2\|\vz\|^2,
\end{equation}
where $\alpha$ and $c_1, c_2$ are nonnegative constants independent of $\vy$ and $\vz$. Then
\begin{align}
-\left(\alpha + 4c_1\|\vy^*\|^2+ 4c_2\|\vz^*\|^2\right)\le f_0(\bar{\vx})-f_0(\vx^*)\le \alpha,\label{eq:pre-rate-obj}\\
\|\vA\bar{\vx}-\vb\| + \big\|[\vf(\bar{\vx})]_+\big\| \le \alpha + c_1\big(1+\|\vy^*\|\big)^2 + c_2\big(1+\|\vz^*\|\big)^2. \label{eq:pre-rate-res}
\end{align}
\end{lemma}

\begin{proof}
Letting $\vy=\vzero$ and $\vz=\vzero$ in \eqref{eq:pre-rate-cond} gives the second inequality in \eqref{eq:pre-rate-obj}.
For any nonnegative $\gamma_y$ and $\gamma_z$, we let 
$$\vy=\gamma_y\frac{\vA\bar{\vx}-\vb}{\|\vA\bar{\vx}-\vb\|},\quad \vz = \gamma_z \frac{[\vf(\bar{\vx})]_+}{\big\|[\vf(\bar{\vx})]_+\big\|}$$ 
and have from \eqref{eq:pre-rate-cond} by using the convention $\frac{0}{0}=0$ that
\begin{align}\label{eq:pre-rate-ineq2}
f_0(\bar{\vx})-f_0(\vx^*) +\gamma_y \|\vA\bar{\vx}-\vb\| + \gamma_z\big\|[\vf(\bar{\vx})]_+\big\|
\le  \alpha + c_1\gamma_y^2+ c_2\gamma_z^2.
\end{align}

Noting 
\begin{equation}\label{eq:pre-rate-ineq3}
-\langle \vy^*, \vA\bar{\vx}-\vb\rangle \ge -\|\vy^*\|\cdot\|\vA\bar{\vx}-\vb\|,\quad - \sum_{i=1}^mz_i^* f_i(\bar{\vx}) \ge - \|\vz^*\|\cdot \big\|[\vf(\bar{\vx})]_+\big\|,
\end{equation} we have from \eqref{eq:opt} and \eqref{eq:pre-rate-ineq2} that
\begin{align*}
(\gamma_y-\|\vy^*\|) \|\vA\bar{\vx}-\vb\| + (\gamma_z -\|\vz^*\|)\|[\vf(\bar{\vx})]_+\big\|
\le  \alpha + c_1\gamma_y^2+ c_2\gamma_z^2
\end{align*}
In the above inequality, letting $\gamma_y=1+\|\vy^*\|$ and $\gamma_z = 1+\|\vz^*\|$ gives \eqref{eq:pre-rate-res}, and letting $\gamma_y=2\|\vy^*\|$ and $\gamma_z = 2\|\vz^*\|$ gives the first inequality in \eqref{eq:pre-rate-obj} by \eqref{eq:opt} and \eqref{eq:pre-rate-ineq3}.
\end{proof}

\subsection{Nesterov's optimal first-order method}
In this subsection, we review Nesterov's optimal first-order method for composite convex programs. The method will be used to approximately solve $\vx$-subproblems in Algorithm \ref{alg:ialm}. It aims at finding a solution of the following problem
\begin{equation}\label{eq:comp-prob}
\Min_\vx \phi(\vx) + \psi(\vx),
\end{equation}
where $\phi$ is a Lipschitz differentiable and strongly convex function with gradient Lipschitz constant $L_\phi$ and strong convexity modulus $\mu\ge0$, and $\psi$ is a simple (possibly nondifferentiable) closed convex function. Algorithm \ref{alg:apg} summarizes the method. Here, for simplicity, we assume $L_\phi$ and $\mu$ are known. The method does not require the value of $L_\phi$ but can estimate it by backtracking. In addition, it only requires a lower estimate of $\mu$; see \cite{nesterov2013gradient} for example.  

\begin{algorithm}\caption{Nesterov's optimal first-order method for \eqref{eq:comp-prob}}\label{alg:apg}
\textbf{Initialization:} choose $\hat{\vx}^0=\vx^0$, $\alpha_0\in (0,1]$, and let $q=\frac{\mu}{L_\phi}$\; 
\For{$k=0,1,\ldots, $}{
Let 
$$\vx^{k+1}=\argmin_{\vx} \langle \nabla \phi(\hat{\vx}^k), \vx\rangle + \frac{L_\phi}{2}\|\vx-\hat{\vx}^k\|^2+\psi(\vx).$$

Set $$\alpha_{k+1}=\frac{q-\alpha_k^2+\sqrt{(q-\alpha_k^2)^2+4\alpha_k^2}}{2},$$ and
$$\hat{\vx}^{k+1}=\vx^{k+1}+\frac{\alpha_k(1-\alpha_k)}{\alpha_k^2+\alpha_{k+1}}(\vx^{k+1}-\vx^k ).$$
}
\end{algorithm}

The theorem below gives the convergence rate of Algorithm \ref{alg:apg} for both convex (i.e., $\mu=0$) and strongly convex (i.e., $\mu>0$) cases; see \cite{nesterov2004introductory, nesterov2013gradient, FISTA2009}. We will use the results to estimate iteration complexity of iALM.
\begin{theorem}\label{thm:rate-apg}
Let $\{\vx^k \}$ be the sequence generated from Algorithm \ref{alg:apg}. Assume $\vx^*$ to be a minimizer of \eqref{eq:comp-prob}. The following results holds:
\begin{enumerate}
\item If $\mu=0$ and $\alpha_0=1$, then
\begin{equation}\label{eq:optimal-rate-cvx}
\phi(\vx^k) + \psi(\vx^k) - \phi(\vx^*)-\psi(\vx^*) \le \frac{2L_\phi\|\vx^0-\vx^*\|^2}{k^2},\,\forall k\ge1.
\end{equation}
\item If $\mu >0$ and $\alpha_0=\sqrt{\frac{\mu}{L_\phi}}$, then
\begin{equation}\label{eq:optimal-rate-scvx}
\phi(\vx^k)-\phi(\vx^*)\le \frac{(L_\phi+\mu)\|\vx^0-\vx^*\|^2}{2}\left(1-\sqrt{\frac{\mu}{L_\phi}}\right)^k,\,\forall k\ge1.
\end{equation}
\end{enumerate}
\end{theorem}

\section{Ergodic convergence rate and iteration complexity results}\label{sec:ergo-rate}
In this section, we first establish an ergodic convergence rate result of Algorithm \ref{alg:ialm}. From that result, we then specify algorithm parameters and estimate the total number of gradient evaluations in order to produce an $\vareps$-optimal solution. Two different settings of the penalty parameters are studied: one with constant penalty and another with geometrically increasing penalty parameters. For each setting, the tolerance error parameter $\vareps_k$ is chosen in an ``optimal'' way so that the total number of gradient evaluation is minimized.
 
Throughout this section, we make the following assumptions. 
\begin{assumption}\label{assump-kkt}
There exists a point $(\vx^*,\vy^*,\vz^*)$ satisfying the KKT conditions in \eqref{eq:kkt}.
\end{assumption}


\begin{assumption}\label{assump-wd-ialm}
For every $k$, there is $\vx^{k+1}$ satisfying \eqref{eq:ialm-x}.
\end{assumption}

The first assumption holds if a certain regularity condition is satisfied, such as the Slater condition (namely, there is an interior point $\vx$ of $\cX$ such that $\vA\vx=\vb$ and $f_i(\vx)<0,\forall i\in [m]$). The second assumption is for the well-definedness of the algorithm. It holds if $\cX$ is compact and $f_i$'s are continuous.

%
%

\subsection{Convergence rate analysis of iALM}
To show the convergence results of Algorithm \ref{alg:ialm}, we first establish a few lemmas.
\begin{lemma}\label{lem:multiplier}
Let $\vy$ and $\vz$ be updated by \eqref{eq:alm-y} and \eqref{eq:alm-z} respectively. Then for any $k$, it holds
\begin{align}
\frac{1}{2\rho_k}\big[\|\vy^{k+1} -\vy\|^2-\|\vy^k -\vy\|^2+\|\vy^{k+1} -\vy^k \|^2\big] - \langle \vy^{k+1} -\vy, \vr^{k+1}\rangle = &~ 0,\label{eq:alm-yterm}\\
\frac{1}{2\rho_k}\big[\|\vz^{k+1} -\vz\|^2-\|\vz^k  -\vz\|^2+\|\vz^{k+1} -\vz^k  \|^2\big]- \sum_{i=1}^m (z_i^{k+1}-z_i)\cdot\max\big(-\frac{z_i^k}{\beta_k}, f_i(\vx^{k+1})\big)=& ~0,\label{eq:alm-zterm}
\end{align}
where $\vr^k =\vA\vx^k -\vb$.
\end{lemma}
\begin{proof}
Using the equality $2\vu^\top \vv = \|\vu\|^2 - \|\vu-\vv\|^2 + \|\vv\|^2$, we have the results from the updates \eqref{eq:alm-y} and \eqref{eq:alm-z}.
\end{proof}
\begin{lemma}\label{lem:ineq-z}
For any $\vz\ge\vzero$, we have
\begin{align}\label{eq:alm-ineq-z}
&\sum_{i=1}^m\big([z_i^k+\beta_k f_i(\vx^{k+1})]_+-z_i\big) f_i(\vx^{k+1})- \sum_{i=1}^m (z_i^{k+1}-z_i)\cdot\max\big(-\frac{z_i^k}{\beta_k}, f_i(\vx^{k+1})\big)\nonumber\\
\ge&\frac{1}{\rho_k^2}(\beta_k-\rho_k)\|\vz^{k+1} -\vz^k  \|^2.
\end{align}
\end{lemma}

\begin{proof}
Denote 
\begin{equation}\label{eq:setIk}
I_+^k=\{i\in [m]: z_i^k + \beta_k f_i(\vx^{k+1}) \ge 0\},\quad I_-^k=[m]\backslash I_+^k.
\end{equation}
 Then 
\begin{align*}
&\text{the left hand side of \eqref{eq:alm-ineq-z}}\\
=&\sum_{i\in I_+^k}\left[(z_i^k-z_i)f_i(\vx^{k+1})+\beta_k[f_i(\vx^{k+1})]^2-\big(z_i^k+\rho_k f_i(\vx^{k+1})-z_i\big)f_i(\vx^{k+1})\right]\\
&+\sum_{i\in I_-^k}\left[-z_i f_i(\vx^{k+1})-\big(z_i^k-\frac{\rho_k z_i^k}{\beta_k}-z_i\big)\big(-\frac{z_i^k}{\beta_k}\big)\right]\\
=&(\beta_k-\rho_k)\sum_{i\in I_+^k}[f_i(\vx^{k+1})]^2+\sum_{i\in I_-^k}\left[-z_i \big(f_i(\vx^{k+1})+\frac{z_i^k}{\beta_k}\big)+\frac{1}{\beta_k^2}(\beta_k-\rho_k)(z_i^k)^2\right]\\
\ge& (\beta_k-\rho_k)\sum_{i\in I_+^k}[f_i(\vx^{k+1})]^2 +\frac{1}{\beta_k^2}(\beta_k-\rho_k)\sum_{i\in I_-^k}(z_i^k)^2\\
=&\frac{1}{\rho_k^2}(\beta_k-\rho_k)\|\vz^{k+1} -\vz^k  \|^2,
\end{align*}
where the inequality follows from $z_i\ge0$ and $f_i(\vx^{k+1})+\frac{z_i^k}{\beta_k}\le 0,\,\forall i\in I_-^k$, and the last equality holds due to the update \eqref{eq:alm-z}.
\end{proof}

The next theorem is a fundamental result by running one iteration of Algorithm \ref{alg:ialm}. 
\begin{theorem}[One-iteration progress of iALM]\label{thm:1iter-ialm}
Let $\{(\vx^k ,\vy^k ,\vz^k  )\}$ be the sequence generated from Algorithm \ref{alg:ialm}. Then for any $\vx\in\cX$ such that $\vA\vx=\vb$ and $f_i(\vx)\le 0,\,\forall i\in [m]$, any $\vy$, and any $\vz\ge\vzero$, it holds that
\begin{align}\label{eq:ialm-ineq2}
&f_0(\vx^{k+1})-f_0(\vx) +\vy^\top \vr^{k+1} + \sum_{i=1}^m z_i f_i(\vx^{k+1})+\frac{\beta_k-\rho_k}{2}\|\vr^{k+1}\|^2\cr
&+\frac{\beta_k-\rho_k}{2\rho_k^2}\|\vz^{k+1} -\vz^k  \|^2+\frac{1}{2\rho_k}\|\vy^{k+1} -\vy\|^2 + \frac{1}{2\rho_k}\|\vz^{k+1} -\vz\|^2\cr
\le & \frac{1}{2\rho_k}\|\vy^k -\vy\|^2+ \frac{1}{2\rho_k}\|\vz^k  -\vz\|^2+\vareps_k.
\end{align}
\end{theorem}


\begin{proof}
From \eqref{eq:ialm-x}, it follows that for any $\vx$ such that $\vA\vx=\vb$,
\begin{equation}\label{eq:ialm-1st}
f_0(\vx^{k+1})+\langle \vy^k , \vr^{k+1}\rangle +\frac{\beta_k}{2}\|\vr^{k+1}\|^2+\Psi_{\beta_k}(\vx^{k+1},\vz^k  )-f_0(\vx)-\Psi_{\beta_k}(\vx,\vz^k  )\le\vareps_k.
\end{equation}
Since $\langle \vy^k , \vr^{k+1}\rangle=\langle \vy^{k+1} -\vy, \vr^{k+1}\rangle + \langle \vy, \vr^{k+1}\rangle -\rho_k\|\vr^{k+1}\|^2$, by adding \eqref{eq:alm-yterm} and \eqref{eq:alm-zterm} to the above inequality, we have 
\begin{align}\label{eq:ialm-ineq1}
&f_0(\vx^{k+1})-f_0(\vx) +\vy^\top \vr^{k+1} + \sum_{i=1}^m z_i f_i(\vx^{k+1})+\sum_{i=1}^m\left([z_i^k+\beta_k f_i(\vx^{k+1})]_+-z_i\right)f_i(\vx^{k+1})\cr
&+\big(\frac{\beta_k}{2}-\rho_k\big)\|\vr^{k+1}\|^2+\Psi_{\beta_k}(\vx^{k+1},\vz^k  )-\sum_{i=1}^m[z_i^k+\beta_k f_i(\vx^{k+1})]_+f_i(\vx^{k+1})-\Psi_{\beta_k}(\vx,\vz^k  )\cr
&+\frac{1}{2\rho_k}\big[\|\vy^{k+1} -\vy\|^2-\|\vy^k -\vy\|^2+\|\vy^{k+1} -\vy^k \|^2\big] \cr
&+\frac{1}{2\rho_k}\big[\|\vz^{k+1} -\vz\|^2-\|\vz^k  -\vz\|^2+\|\vz^{k+1} -\vz^k  \|^2\big]- \sum_{i=1}^m (z_i^{k+1}-z_i)\cdot\max\big(-\frac{z_i^k}{\beta_k}, f_i(\vx^{k+1})\big)\cr
\le & \vareps_k.
\end{align}

Note that
\begin{align}\label{eq:ialm-ineq3}
&\Psi_{\beta_k}(\vx^{k+1},\vz^k  )-\sum_{i=1}^m[z_i^k+\beta_k f_i(\vx^{k+1})]_+f_i(\vx^{k+1})\cr
=&\sum_{i\in I_+^k}\left[z_i^k f_i(\vx^{k+1}) +\frac{\beta_k}{2}[f_i(\vx^{k+1})]^2-[z_i^k +\beta_k f_i(\vx^{k+1})]f_i(\vx^{k+1})\right]+\sum_{i\in I_-^k}\left[-\frac{(z_i^k)^2}{2\beta_k}\right]\cr
=&-\sum_{i\in I_+^k}\frac{\beta_k}{2}[f_i(\vx^{k+1})]^2-\sum_{i\in I_-^k}\frac{(z_i^k)^2}{2\beta_k}\cr
=&-\frac{\beta_k}{2\rho_z^2}\|\vz^{k+1} -\vz^k  \|^2,
\end{align}
where the sets $I_+^k$ and $I_-^k$ are defined in \eqref{eq:setIk}.
In addition, if $f_i(\vx)\le 0,\,\forall i\in [m]$, then $\Psi_{\beta_k}(\vx,\vz^k  )\le 0.$ Hence, plugging \eqref{eq:alm-ineq-z} and \eqref{eq:ialm-ineq3} into \eqref{eq:ialm-ineq1} yields \eqref{eq:ialm-ineq2}.
\end{proof}

By Lemma \ref{lem:pre-rate} and Theorem \ref{thm:1iter-ialm}, we have the following convergence rate estimate of Algorithm \ref{alg:ialm}.
\begin{theorem}[Ergodic convergence rate of iALM]\label{thm:ialm-rate}
Under Assumptions \ref{assump-kkt} and \ref{assump-wd-ialm}, let $\{(\vx^k ,\vy^k ,\vz^k  )\}$ be the sequence generated from Algorithm \ref{alg:ialm} with $\vy^0=\vzero, \vz^0 = \vzero$ and $0<\rho_k\le\beta_k,\,\forall k$. Then
\begin{subequations}\label{eq:ialm-rate}
\begin{align}
\big|f_0(\bar{\vx}^{K})-f_0(\vx^*)\big| \le \frac{1}{\sum_{t=0}^{K-1}\rho_t}\left(2\|\vy^*\|^2+ 2\|\vz^*\|^2+\sum_{k=0}^{K-1}\rho_k\vareps_k\right),\label{eq:ialm-rate-obj}\\
\|\vA\bar{\vx}^{K}-\vb\| + \big\| [\vf(\bar{\vx}^{K})]_+\big\| \le  \frac{1}{\sum_{t=0}^{K-1}\rho_t}\left(\frac{(1+\|\vy^*\|)^2}{2}+ \frac{(\|1+\|\vz^*\|)^2}{2}+\sum_{k=0}^{K-1}\rho_k\vareps_k\right),\label{eq:ialm-rate-res}
\end{align}
\end{subequations}
where 
\begin{equation}\label{eq:barx-K}
\bar{\vx}^{K}=\frac{\sum_{t=0}^{K-1}\rho_t\vx^{t+1} }{\sum_{t=0}^{K-1}\rho_t}.
\end{equation}
\end{theorem}

\begin{proof}
Since $\rho_k\le \beta_k$, the two terms $\frac{\beta_k-\rho_k}{2}\|\vr^{k+1}\|^2$ and 
$\frac{\beta_k-\rho_k}{2\rho_k^2}\|\vz^{k+1} -\vz^k  \|^2$ are nonnegative. Dropping them and multiplying $\rho_k$ to both sides of \eqref{eq:ialm-ineq2} yields 
\begin{align}\label{eq:ialm-ineq2-1}
&\rho_k\left[f_0(\vx^{k+1})-f_0(\vx) +\vy^\top  \vr^{k+1} + \sum_{i=1}^m z_i f_i(\vx^{k+1})\right]+\frac{1}{2}\|\vy^{k+1} -\vy\|^2 + \frac{1}{2}\|\vz^{k+1} -\vz\|^2\cr
\le & \frac{1}{2}\|\vy^k -\vy\|^2+ \frac{1}{2}\|\vz^k  -\vz\|^2+\rho_k\vareps_k.
\end{align}

Summing up \eqref{eq:ialm-ineq2} with $\vx=\vx^*$ gives
\begin{align}\label{eq:ialm-ineq2-2}
&\sum_{k=0}^{K-1}\rho_k\left[f_0(\vx^{k+1})-f_0(\vx^*) +\vy^\top \vr^{k+1} + \sum_{i=1}^m z_i f_i(\vx^{k+1})\right]+\frac{1}{2}\|\vy^{K}-\vy\|^2 + \frac{1}{2}\|\vz^{K}-\vz\|^2\cr
\le & \frac{1}{2}\|\vy^0-\vy\|^2+ \frac{1}{2}\|\vz^0-\vz\|^2+\sum_{k=0}^{K-1}\rho_k\vareps_k.
\end{align}
By the convexity of $f_i$'s and the nonnegativity of $\vz$, we have
\begin{align*}
&~f_0(\bar{\vx}^{K})-f_0(\vx^*) +\vy^\top  (\vA\bar{\vx}^{K}-\vb) + \sum_{i=1}^m z_i f_i(\bar{\vx}^{K})\\
\le &~\frac{1}{\sum_{t=0}^{K-1}\rho_t} \sum_{k=0}^{K-1}\rho_k\left[f_0(\vx^{k+1})-f_0(\vx^*) +\vy^\top  \vr^{k+1} + \sum_{i=1}^m z_i f_i(\vx^{k+1})\right],
\end{align*}
which together with \eqref{eq:ialm-ineq2-2} implies
$$f_0(\bar{\vx}^{K})-f_0(\vx^*) +\vy^\top  (\vA\bar{\vx}^{K}-\vb) + \sum_{i=1}^m z_i f_i(\bar{\vx}^{K})
\le  \frac{1}{\sum_{t=0}^{K-1}\rho_t}\left(\frac{1}{2}\|\vy\|^2+ \frac{1}{2}\|\vz\|^2+\sum_{k=0}^{K-1}\rho_k\vareps_k\right).$$
The results thus follow from Lemma \ref{lem:pre-rate} with 
$$\alpha=\frac{\sum_{k=0}^{K-1}\rho_k\vareps_k}{\sum_{k=0}^{K-1}\rho_k},\ c_1=\frac{1}{2\sum_{k=0}^{K-1}\rho_k},\ c_2=\frac{1}{2\sum_{k=0}^{K-1}\rho_k},$$
and we complete the proof.
\end{proof}


\begin{remark}
Note that if $\rho_k\equiv\rho>0,\,\forall k$ and $\{\vareps_k\}$ is summable, then a sublinear convergence result follows from \eqref{eq:ialm-rate}. The work \cite{rockafellar1973multiplier} has also analyzed the convergence of Algorithm \ref{alg:ialm} through the augmented dual function $d_\beta$. However, it requires $\sum_{k=1}^\infty \sqrt{\vareps_k} < \infty$, which is strictly stronger than the condition $\sum_{k=1}^\infty \vareps_k < \infty$.
\end{remark}


\subsection{Iteration complexity of iALM}
In this subsection, we apply Nesterov's optimal first-order method to each $\vx$-subproblem \eqref{eq:ialm-x} and estimate the total number of gradient evaluations to produce an $\vareps$-optimal solution of \eqref{eq:ccp}. 
Note that the convergence rate results in Theorem \ref{thm:ialm-rate} do not assume specific structures of \eqref{eq:ccp} except convexity. If the problem \eqref{eq:ccp} has richer structures than those in \eqref{eq:lip-gf}, 
more efficient methods can be applied to the subproblems in \eqref{eq:ialm-x}. 

The following results are easy to show from the Lipschitz differentiability of $g$ and $f_i,\, i\in [m]$.
\begin{proposition}\label{prop-lip}
Assume \eqref{eq:lip-g}, \eqref{eq:lip-f}, and the boundedness of $\dom(h)\cap \cX$. Then there exist constants $B_1,\ldots, B_m$ such that
\begin{subequations}\label{eq:lip}
\begin{align}
\max\big(|f_i(\vx)|, \|\nabla f_i(\vx)\|\big) \le B_i, \,& \forall \vx\in \dom(h)\cap\cX, \forall i\in [m],\label{eq:bd-f}\\
|f_i(\hat{\vx})-f_i(\tilde{\vx})| \le B_i\|\hat{\vx}-\tilde{\vx}\|, \, & \forall \hat{\vx}, \tilde{\vx} \in \dom(h)\cap\cX, \forall i\in [m].\label{eq:lip-fv}
\end{align}
\end{subequations}
\end{proposition}

Let the smooth part of $\cL_\beta$ be denoted as
$$F_\beta(\vx,\vy,\vz) = \cL_\beta(\vx,\vy,\vz) - h(\vx).$$
Based on \eqref{eq:lip}, we are able to show Lipschitz continuity of $\nabla_\vx F_\beta(\vx,\vy,\vz)$ with respect to $\vx$ for every $(\vy,\vz)$.
\begin{lemma}\label{lem:lip}
Assume \eqref{eq:lip-g}, \eqref{eq:lip-f}, and the boundedness of $\dom(h)\cap \cX$. Let $B_i$'s be given in Proposition \ref{prop-lip}. Then $\nabla_\vx F_{\beta_k}(\vx,\vy^k,\vz^k)$ is Lipschitz continuous on $\dom(h)\cap\cX$ in terms of $\vx$ with constant
\begin{equation}\label{eq:def-Lz}
L(\vz^k) = L_0+\beta_k\|\vA^\top \vA\|+\sum_{i=1}^m \big(\beta_k B_i(B_i+L_i)+L_i|z_i^k|\big).
\end{equation}
\end{lemma}
\begin{proof}
For ease of description, we let $\beta=\beta_k$ and $(\vy,\vz)=(\vy^k,\vz^k)$ in the proof. First we notice that $\frac{\partial}{\partial u}\psi_\beta(u,v)=[\beta u+v]_+$, and thus for any $v$,
$$\left|\frac{\partial}{\partial u}\psi_\beta(\hat{u},v)-\frac{\partial}{\partial u}\psi_\beta(\tilde{u},v)\right|\le \beta|\hat{u}-\tilde{u}|,\,\forall \hat{u},\tilde{u}.$$
Let $h_i(x,z_i) = \psi_\beta(f_i(\vx),z_i),\,i=1,\ldots,m$. Then
\begin{align*}
&\|\nabla_\vx h_i(\hat{\vx},z_i)-\nabla_\vx h_i(\tilde{\vx},z_i)\|\\
=&\big\|\frac{\partial}{\partial u}\psi_\beta(f_i(\hat{\vx}),z_i)\nabla f_i(\hat{\vx})-\frac{\partial}{\partial u}\psi_\beta(f_i(\tilde{\vx}),z_i)\nabla f_i(\tilde{\vx})\big\|\\
\le&\big\|\frac{\partial}{\partial u}\psi_\beta(f_i(\hat{\vx}),z_i)\nabla f_i(\hat{\vx})-\frac{\partial}{\partial u}\psi_\beta(f_i(\tilde{\vx}),z_i)\nabla f_i(\hat{\vx})\big\|\\
&+\big\|\frac{\partial}{\partial u}\psi_\beta(f_i(\tilde{\vx}),z_i)\nabla f_i(\hat{\vx})-\frac{\partial}{\partial u}\psi_\beta(f_i(\tilde{\vx}),z_i)\nabla f_i(\tilde{\vx})\big\|\\
\le& \beta|f_i(\hat{\vx})-f_i(\tilde{\vx})|\cdot\|\nabla f_i(\hat{\vx})\|+\big|\frac{\partial}{\partial u}\psi_\beta(f_i(\tilde{\vx}),z_i)\big|\cdot\|\nabla f_i(\hat{\vx})-\nabla f_i(\tilde{\vx})\|\\
\le & \beta B_i^2\|\hat{\vx}-\tilde{\vx}\|+ L_i(\beta B_i + |z_i|) \|\hat{\vx}-\tilde{\vx}\|.
\end{align*}
Hence,
\begin{align*}
&\|\nabla_\vx F_\beta(\hat{\vx},\vy,\vz)-\nabla_\vx F_\beta(\tilde{\vx},\vy,\vz)\|\\
\le& \|\nabla g(\hat{\vx})-\nabla g(\tilde{\vx})\|+\beta\|\vA^\top \vA(\hat{\vx}-\tilde{\vx})\| + \sum_{i=1}^m \|\nabla_\vx h_i(\hat{\vx},z_i)-\nabla_\vx h_i(\tilde{\vx},z_i)\|\\
\le & \left(L_0+\beta\|\vA^\top \vA\|+\sum_{i=1}^m \big[\beta B_i^2+ L_i(\beta B_i + |z_i|)\big]\right)\|\hat{\vx}-\tilde{\vx}\|,
\end{align*}
which completes the proof.
\end{proof}

Therefore, letting
$\phi(\vx) = F_{\beta_k}(\vx,\vy^k ,\vz^k  )$ and $\psi(\vx) = h(\vx)+\iota_\cX(\vx),$
we can apply Nesterov's optimal first-order method in Algorithm \ref{alg:apg} to find $\vx^{k+1}$ in \eqref{eq:ialm-x}. From Theorem \ref{thm:rate-apg}, we have the following results. Note that if the strong convexity constant $\mu=0$, the problem is just convex.

\begin{lemma}\label{lem:sub-iter}
Assume that $g$ is strongly convex with modulus $\mu\ge0$. Given $\vareps_k>0$, if we start from $\vx^k $ and run Algorithm \ref{alg:apg}, then at most $t_k$ iterations are needed to produce $\vx^{k+1}$ such that \eqref{eq:ialm-x} holds,
where 
\begin{equation}\label{eq:tk-cvx}
t_k = \left\{
\begin{array}{ll}
\left\lceil\dfrac{\dist(\vx^k ,\cX_k^*)\sqrt{2L(\vz^k  )}}{\sqrt{\vareps_k}}\right\rceil, & \text{ if }\mu=0,\\[0.5cm]
\left\lceil\dfrac{\log\left(\frac{L(\vz^k  )+\mu}{2\vareps_k}[\dist(\vx^k ,\cX_k^*)]^2\right)}{\log1/\big(1-\sqrt{\frac{\mu}{L(\vz^k  )}}\big)}\right\rceil, & \text{ if }\mu>0,
\end{array}
\right.
\end{equation}
and $\cX_k^*$ denotes the set of optimal solutions to $\min_{\vx\in \cX}\cL_{\beta_k}(\vx,\vy^k ,\vz^k  )$.
 
\end{lemma}

Below we specify the sequences $\{\beta_k\}, \{\rho_k\}$ and $\{\vareps_k\}$ for a given $\vareps>0$, and through combining Theorem \ref{thm:ialm-rate} and Lemma \ref{lem:sub-iter}, we give the iteration complexity results of iALM for producing an $\vareps$-optimal solution. We study two cases. In the first case, a constant penalty parameter is used, 
and in the second case, we geometrically increase $\beta_k$ and $\rho_k$. 

Given $\vareps>0$, we set $\{\beta_k\}$ and $\{\rho_k\}$ according to one of the follows:

\begin{setting}[constant penalty]\label{set-penalty} Let $K$ be a positive integer number and $C_\beta$ a positive real number. Set $$\rho_k=\beta_k=\frac{C_\beta}{K\vareps},\,\forall \, 0\le k < K.$$ 
\end{setting}

\begin{setting}[geometrically increasing penalty]\label{set-ialm} Let $K$ be a positive integer number, $C_\beta$ a positive real number, and $\sigma>1$. Set 
\begin{equation}\label{eq:def-beta-g}
\beta_g = \frac{C_\beta}{\vareps}\frac{\sigma-1}{\sigma^K-1},
\end{equation}
and
$$ \rho_k=\beta_k=\beta_g\sigma^k, \,\forall\, 0\le k < K.$$ 
\end{setting}

Note that if $K=1$, the above two settings are the same, and in this case, Algorithm \ref{alg:ialm} simply reduces to the quadratic penalty method. For either of the above two settings, we have $\sum_{k=0}^{K-1}\rho_k=\frac{C_\beta}{\vareps}.$ 


From \eqref{eq:def-Lz}, we see that the Lipschitz constant depends on $\vz^k$. Hence, from \eqref{eq:tk-cvx}, to solve the $\vx$-subproblem to the accuracy $\vareps_k$, the number of gradient evaluations will depend on $\vz^k  $. Below we show that if $\vareps_k$ is sufficiently small, $\vz^k$ can be bounded and thus so is $L(\vz^k)$. 

\begin{lemma}\label{lem:bd-z-Lz}
Let $\{(\vx^k,\vy^k,\vz^k)\}_{k=0}^K$ be the sequence generated from Algorithm \ref{alg:ialm} with $\{\beta_k\}$ and $\{\rho_k\}$ set according to either Setting \ref{set-penalty} or Setting \ref{set-ialm}. If $\vy^0=\vzero, \vz^0=\vzero$, and $\vareps_k$'s are chosen such that 
\begin{equation}\label{eq:req-eps}
\sum_{k=0}^{K-1}\rho_k\vareps_k\le \frac{C_\vareps}{2},
\end{equation}
 then
\begin{align}\label{eq:bd-Lzk}
L(\vz^k)\le L_*+\beta_k H,\,\forall \, 0\le k\le K,
\end{align}
where 
$$ H=\|\vA^\top \vA\|+ \sum_{i=1}^m  B_i(B_i+L_i),\ L_*=L_0+ \|\vell\| \left(\|\vy^*\|+2\|\vz^*\|+\sqrt{C_\vareps}\right)$$
and $\vell$ is given in \eqref{eq:vec-ell}.
\end{lemma}

\begin{proof}
Letting $(\vx,\vy,\vz)=(\vx^*,\vy^*,\vz^*)$ in \eqref{eq:ialm-ineq2-1} and using \eqref{eq:opt}, we have 
$$\frac{1}{2}\|\vy^{k+1} -\vy^*\|^2 + \frac{1}{2}\|\vz^{k+1} -\vz^*\|^2
\le  \frac{1}{2}\|\vy^k -\vy^*\|^2+ \frac{1}{2}\|\vz^k  -\vz^*\|^2+\rho_k\vareps_k.$$
Summing the above inequality yields 
$$\frac{1}{2}\|\vy^{k}-\vy^*\|^2 + \frac{1}{2}\|\vz^{k}-\vz^*\|^2
\le  \frac{1}{2}\|\vy^0-\vy^*\|^2+ \frac{1}{2}\|\vz^0-\vz^*\|^2+\sum_{t=0}^{k-1}\rho_t\vareps_t, \, \forall \, 0\le k \le K,$$
which implies
$$\|\vz^k  \| \le \|\vz^*\|+\sqrt{\|\vy^0-\vy^*\|^2+\|\vz^0-\vz^*\|^2+2\sum_{t=0}^{k-1}\rho_t\vareps_t}.$$
Since $\|\vu\|\le \|\vu\|_1$ for any vector $\vu$, we have from the above inequality that
\begin{equation}\label{eq:bd-zk0}
\|\vz^k  \| \le  \|\vz^*\|+\|\vy^0-\vy^*\|+\|\vz^0-\vz^*\|+\sqrt{2\sum_{t=0}^{K-1}\rho_t\vareps_t},\, \forall \, 0\le k \le K.
\end{equation}

Hence, if $\vy^0=\vzero$ and $\vz^0=\vzero$, and \eqref{eq:req-eps} holds, it follows from the above inequality that
\begin{equation}\label{eq:bd-zk}
\|\vz^k  \| \le  \|\vy^*\|+2\|\vz^*\|+\sqrt{C_\vareps},\, \forall \, 0\le k \le K,
\end{equation}
By the Cauchy-Schwartz inequality, we have from \eqref{eq:def-Lz} that for any $0\le k\le K$,
$$L(\vz^k  ) \le L_0+\beta_k H + \|\vz^k  \|\cdot\|\vell\|$$
which together with \eqref{eq:bd-zk} gives the result in \eqref{eq:bd-Lzk}.
\end{proof}

\textbf{``Optimal'' subproblem accuracy parameters.} If $t_k$ gradient evaluations are required to produce $\vx^{k+1}$, then the total number of gradient evaluations is 
$T_K=\sum_{k=0}^{K-1}t_k$ to generate $\{\vx^k\}_{k=1}^K$. Given $\vareps>0$, and $\{\beta_k\}$,$\{\rho_k\}$ set according to either Setting \ref{set-penalty} or Setting \ref{set-ialm}, we can choose $\{\vareps_k\}$ to minimize $T_K$ subject to the condition in \eqref{eq:req-eps}. 

When $\mu=0$, we solve the following problem:
\begin{equation*}
\Min_{\veps>\vzero} \sum_{k=0}^{K-1}\frac{\dist(\vx^k,\cX_k^*)\sqrt{L(\vz^k)}}{\sqrt{\vareps_k}}, \st \sum_{k=0}^{K-1}\beta_k\vareps_k \le \frac{C_\vareps}{2},
\end{equation*}
where $\veps=[\vareps_0,\ldots,\vareps_{K-1}].$ Through formulating the KKT system of the above problem, one can easily find the optimal $\veps$ given by
\begin{equation}\label{eq:best-eps-cvx}
\vareps_k=\frac{C_\vareps}{2}\frac{[\dist(\vx^k,\cX_k^*)]^{\frac{2}{3}}[L(\vz^k)]^{\frac{1}{3}}}{\beta_k^{\frac{2}{3}}\sum_{t=0}^{K-1}\beta_t^{\frac{1}{3}}[\dist(\vx^t,\cX_t^*)]^{\frac{2}{3}}[L(\vz^t)]^{\frac{1}{3}}},\,\forall \, 0\le k<K.
\end{equation}

When $\mu>0$, we solve the problem below:
\begin{equation}\label{eq:prob-find-eps-scvx}
\Min_{\veps>\vzero}\sum_{k=0}^{K-1} \sqrt{\frac{L(\vz^k)}{\mu}}\log\left(\frac{L(\vz^k)+\mu}{2\vareps_k}[\dist(\vx^k,\cX_k^*)]^2\right),\st \sum_{k=0}^{K-1}\beta_k\vareps_k \le \frac{C_\vareps}{2},
\end{equation}
whose optimal solution is given by
\begin{equation}\label{eq:best-eps-scvx}
\vareps_k=\frac{C_\vareps}{2}\frac{\sqrt{L(\vz^k)}}{\beta_k\sum_{t=0}^{K-1}\sqrt{L(\vz^t)}},\,\forall \, 0\le k<K.
\end{equation}
Note that the summand in the objective of \eqref{eq:prob-find-eps-scvx} is not exactly the same as that in the second inequality of \eqref{eq:tk-cvx}. They are close if $\mu\ll L(\vz^k)$ since $\log(1+a)=a + o(a)$. 

The optimal $\vareps_k$ given in \eqref{eq:best-eps-cvx} and \eqref{eq:best-eps-scvx} depends on $\dist(\vx^k,\cX_k^*)$ and the future points $\vz^{k+1},\ldots,\vz^{K-1}$, which are unknown at iteration $k$. We do not assume these unknowns. Instead, we set $\vareps_k$ in two different ways. One way is to simply set \begin{equation}\label{eq:const-eps}
\vareps_k=\frac{\vareps }{2}\frac{C_\vareps}{C_\beta},\,\forall \, 0\le k<K,
\end{equation}
for both cases of $\mu=0$ and $\mu>0$. Another way is to
let
\begin{equation}\label{eq:best-eps-cvx-app}
\vareps_k=\frac{C_\vareps}{2}\frac{1}{\beta_k^{\frac{1}{3}}\sum_{t=0}^{K-1}\beta_t^{\frac{2}{3}}},\,\forall \, 0\le k<K,
\end{equation}
for the case of $\mu=0$,
and
\begin{equation}\label{eq:best-eps-scvx-app}
\vareps_k=\frac{C_\vareps}{2}\frac{1}{\sqrt{\beta_k}\sum_{t=0}^{K-1}\sqrt{\beta_t}},\,\forall \, 0\le k<K,
\end{equation}
for the case of $\mu>0$.
We see that if $\beta_k H$ dominates $L_*$ and $\dist(\vx^k,\cX_k^*)$ is roughly the same for all $k$'s, then $\{\vareps_k\}$ in \eqref{eq:best-eps-cvx-app} and \eqref{eq:best-eps-scvx-app} well approximate those in \eqref{eq:best-eps-cvx} and \eqref{eq:best-eps-scvx}. If $\{\beta_k\}$ and $\{\rho_k\}$ are set according to Setting \ref{set-penalty}, i.e., constant parameters, then the $\{\vareps_k\}$ in both \eqref{eq:best-eps-cvx-app} and \eqref{eq:best-eps-scvx-app} is constant as in \eqref{eq:const-eps}. 

Plugging these parameters into \eqref{eq:tk-cvx}, we have the following estimates on the total number of gradient evaluations.

\begin{theorem}[Iteration complexity with constant penalty and constant error]\label{thm:it-comp-const}
For any given $\vareps>0$, let $K$ be a positive integer number and $C_\beta, C_\vareps$ two positive real numbers. Set $\{\beta_k\}$ and $\{\rho_k\}$ according to Setting \ref{set-penalty} and $\{\vareps_k\}$ by \eqref{eq:const-eps}. Let $\bar{\vx}^K$ be given in \eqref{eq:barx-K}. Then 
\begin{subequations}\label{eq:error-barx-K}
\begin{align}
\big|f_0(\bar{\vx}^{K})-f_0(\vx^*)\big| \le \frac{\vareps\big(2\|\vy^*\|^2+ 2\|\vz^*\|^2\big)}{C_\beta}+\frac{\vareps}{2}\frac{C_\vareps}{C_\beta},\label{eq:error-barx-K-obj}\\
\|\vA\bar{\vx}^{K}-\vb\| + \big\| [\vf(\bar{\vx}^{K})]_+ \big\|\le  \frac{\vareps\big[(1+\|\vy^*\|)^2+ (1+\|\vz^*\|)^2\big]}{2C_\beta}+\frac{\vareps}{2}\frac{C_\vareps}{C_\beta}.\label{eq:error-barx-K-res}
\end{align}
\end{subequations}
Assume $\mu\le \frac{L_0}{4}$. Then Algorithm \ref{alg:ialm} can produce $\bar{\vx}^K$ by evaluating gradients of $g, f_i, \,i\in [m]$ in at most $T_K$ times, where
\begin{equation}\label{eq:T-penaly-cvx}
T_K\le \left\lceil2DK\sqrt{\frac{C_\beta}{C_\vareps}}\left(\sqrt{\frac{L_*}{\vareps}}+\frac{1}{\vareps}\sqrt{\frac{C_\beta H}{K}}\right)+K\right\rceil, \text{ if }\mu=0,
\end{equation}
and
\begin{equation}\label{eq:T-penaly-strcvx}
T_K\le \left\lceil2K\left(\sqrt{\frac{L_*}{\mu}}+\sqrt{\frac{C_\beta H}{\mu K\vareps}}\right)\log\left(\frac{D^2 C_\beta}{C_\vareps}\big(\frac{L_*+\mu}{\vareps}+\frac{C_\beta H}{K\vareps^2}\big)\right) + K\right\rceil, \text{ if }\mu>0.
\end{equation}

\end{theorem}
\begin{proof}
The results in \eqref{eq:error-barx-K} directly follows from Theorem \ref{thm:ialm-rate} and the settings of $\{\beta_k\}$, $\{\rho_k\}$, and $\{\vareps_k\}$. For the total number of gradient evaluations, we use the inequalities in \eqref{eq:tk-cvx}. First, for the case of $\mu=0$, from the first inequality of \eqref{eq:tk-cvx} and the parameter setting, it follows that the total number of gradient evaluations
\begin{equation}\label{eq:total-grad-cvx-adp1}
T_K\le \sum_{k=0}^{K-1}\frac{\dist(\vx^k ,\cX_k^*)\sqrt{2(L_*+\frac{C_\beta H}{K\vareps})}}{\sqrt{\vareps/2}\sqrt{C_\vareps/C_\beta}}+K.
\end{equation}
Since $\sqrt{a+b}\le \sqrt{a}+\sqrt{b}$ for any two nonnegative numbers $a,b$, we have from the above inequality and by noting $\dist(\vx^k ,\cX_k^*)\le D$ that
\begin{align*}
T_K \le 2D\sqrt{\frac{C_\beta}{C_\vareps}}\sum_{k=0}^{K-1}\frac{\sqrt{L_*}+\sqrt{\frac{C_\beta H}{K\vareps}}}{\sqrt{\vareps}}+K=2DK\sqrt{\frac{C_\beta}{C_\vareps}}\left(\sqrt{\frac{L_*}{\vareps}}+\frac{1}{\vareps}\sqrt{\frac{C_\beta H}{K}}\right)+K,
\end{align*}
which gives \eqref{eq:T-penaly-cvx}.

For the case of $\mu>0$, we first note that for any $0<a\le1$, it holds $\log(1+a)\ge a-\frac{a^2}{2}\ge\frac{a}{2}$. Hence, if $\mu\le \frac{L_0}{4}$, we have $\mu\le \frac{L(\vz^k)}{4}$ and $\frac{\sqrt{\mu/L(\vz^k)}}{1-\sqrt{\mu/L(\vz^k)}}\le 1$. Therefore,
$$\log\frac{1}{1-\sqrt{\mu/L(\vz^k)}}=\log\left(1+\frac{\sqrt{\mu/L(\vz^k)}}{1-\sqrt{\mu/L(\vz^k)}}\right)\ge \frac{1}{2}\frac{\sqrt{\mu/L(\vz^k)}}{1-\sqrt{\mu/L(\vz^k)}},$$ and thus
\begin{equation}\label{eq:bd-denom}
\frac{1}{\log\frac{1}{1-\sqrt{\mu/L(\vz^k)}}}\le 2\sqrt{\frac{L(\vz^k)}{\mu}}\left(1-\sqrt{\mu/L(\vz^k)}\right) \le 2\sqrt{\frac{L(\vz^k)}{\mu}}.
\end{equation} Using the above inequality and the second inequality of \eqref{eq:tk-cvx}, we have that the total number of gradient evaluations
\begin{equation}
T_K\le  \sum_{k=0}^{K-1}2\sqrt{\frac{L_*+\frac{C_\beta H}{K\vareps}}{\mu}}\log\left(\frac{L_*+\frac{C_\beta H}{K\vareps}+\mu}{\vareps C_\vareps/C_\beta}[\dist(\vx^k ,\cX_k^*)]^2\right)+K.\label{eq:total-grad-scvx-adp1}
\end{equation}
Since $\sqrt{L_*+\frac{C_\beta H}{K\vareps}}\le \sqrt{L_*}+\sqrt{\frac{C_\beta H}{K\vareps}}$ and $\dist(\vx^k ,\cX_k^*)\le D$, the above inequality implies \eqref{eq:T-penaly-strcvx}.
This completes the proof.
\end{proof}

We make two observations below about the results in Theorem \ref{thm:it-comp-const}.
\begin{remark}\label{rm:iter-to-rate}
From the error bounds in \eqref{eq:error-barx-K}, we see that if 
\begin{equation}\label{eq:bd-C-beta1}
2 C_\beta \ge\max\big(4\|\vy^*\|^2+4\|\vz^*\|^2, (1+\|\vy^*\|)^2+(1+\|\vz^*\|)^2\big) + C_\vareps,
\end{equation} 
then $\bar{\vx}^K$ is an $\vareps$-optimal solution. Otherwise, the errors in \eqref{eq:error-barx-K} are multiples of $\vareps$. 
If we represent $\vareps$ by the total number $t$ of gradient evaluations, we can obtain the convergence rate result in terms of $t$. Let $C_\beta=C_\vareps$ and $K=1$ in \eqref{eq:T-penaly-cvx}. Then the total number of gradient evaluations is about $t = 2D\left(\sqrt{\frac{L_*}{\vareps}}+\frac{1}{\vareps}\sqrt{C_\beta H}\right).$
By quadratic formula, one can easily show  that
$$\vareps = \frac{\left(D\sqrt{L_*}+\sqrt{L_*D^2+2Dt\sqrt{C_\beta H}}\right)^2}{t^2}\le \frac{4L_*D^2}{t^2}+\frac{4D\sqrt{C_\beta H}}{t}.$$
Let $\hat{\vx}^t=\bar{\vx}^K$ to specify the dependence of the iterate on the number of gradient evaluations. Plugging the above $\vareps$ into \eqref{eq:error-barx-K}, we have
\begin{subequations}\label{eq:error-barx-K-rate}
\begin{align}
\big|f_0(\hat{\vx}^{t})-f_0(\vx^*)\big| \le \left(\frac{2\|\vy^*\|^2+ 2\|\vz^*\|^2}{C_\beta}+\frac{1}{2}\right)\left(\frac{4L_*D^2}{t^2}+\frac{4D\sqrt{C_\beta H}}{t}\right),\label{eq:error-barx-K-rate-obj}\\
\|\vA\hat{\vx}^{t}-\vb\| + \big\| [\vf(\hat{\vx}^{t})]_+ \big\|\le  \left(\frac{(1+\|\vy^*\|)^2+ (1+\|\vz^*\|)^2}{2C_\beta}+\frac{1}{2}\right)\left(\frac{4L_*D^2}{t^2}+\frac{4D\sqrt{C_\beta H}}{t}\right).\label{eq:error-barx-K-rate-res}
\end{align}
\end{subequations}
If there are no equality or inequality constraints, then $H=0$, $\vy^*=\vzero, \vz^*=\vzero$, and the rate of convergence in \eqref{eq:error-barx-K-rate-obj} matches with the optimal one in \eqref{eq:optimal-rate-cvx}; if the objective $f_0(\vx)\equiv 0$ and there are no inequality constraints, then $H=\|\vA^\top\vA\|$, $\vy^*=\vzero, \vz^*=\vzero$, $L_*=0$, and the rate of convergence with $C_\beta=2$ in \eqref{eq:error-barx-K-rate-res} roughly becomes 
$$\|\vA\hat{\vx}^t-\vb\|^2\le \frac{8\sqrt{2}D^2\|\vA^\top\vA\|}{t^2},$$
whose order is also optimal. Therefore, the order of convergence rate in \eqref{eq:error-barx-K-rate} is optimal, and so is the iteration complexity result in \eqref{eq:T-penaly-cvx}.

For the strongly convex case, if there are no equality or inequality constraints, the iteration complexity result in \eqref{eq:T-penaly-strcvx} is optimal by comparing it to \eqref{eq:optimal-rate-scvx}. With the existence of constraints and nonsmooth term in the objective, the order $1/\sqrt{\vareps}$ also appears to be optimal and is the best we can find in the literature; see the discussion in section \ref{sec:review}.

\end{remark}

\begin{remark}\label{remark:penalty-comp}
From both \eqref{eq:T-penaly-cvx} and \eqref{eq:T-penaly-strcvx}, we see that $T_1\le T_K,\,\forall K\ge 1$, i.e., $K=1$ is the best. Note that if $\vy^0=\vzero,\vz^0=\vzero$, and $K=1$, Algorithm \ref{alg:ialm} reduces to the quadratic penalty method by solving a single penalty problem. However, practically $K>1$ could be better since $\dist(\vx^k,\cX_k^*)$ usually decreases as $k$ increases. Hence, from \eqref{eq:total-grad-cvx-adp1} or \eqref{eq:total-grad-scvx-adp1}, $T_K$ can be smaller than $T_1$ if $K>1$; see our numerical results in Table \ref{table:cvx-med-Lip}. 
\end{remark}

\begin{theorem}[Iteration complexity with geometrically increasing penalty and constant error]\label{thm:iter-comp-geo-const-eps}
For any given $\vareps>0$, let $K$ be a positive integer number and $C_\beta, C_\vareps$ two positive real numbers. Set $\{\beta_k\}$ and $\{\rho_k\}$ according to Setting \ref{set-ialm} and $\{\vareps_k\}$ to \eqref{eq:const-eps}.  Assume $\mu\le \frac{L_0}{4}$. Let $\bar{\vx}^K$ be given in \eqref{eq:barx-K}. Then the inequalities in \eqref{eq:error-barx-K} hold, and Algorithm \ref{alg:ialm} can produce $\bar{\vx}^K$ by evaluating gradients of $g, f_i, \,i\in [m]$ in at most $T_K$ times, where  
\begin{equation}\label{eq:T-ialm-cvx-const-eps}
T_K\le \left\lceil 2D \sqrt{\frac{C_\beta}{C_\vareps}}\left(K\sqrt{\frac{L_*}{\vareps}}+\frac{\sqrt{C_\beta H(\sigma-1)}}{\vareps(\sqrt{\sigma}-1)}\right)+K\right\rceil, \text{ if }\mu=0,
\end{equation}
and
\begin{equation}\label{eq:T-ialm-strcvx-const-eps}
T_K\le \left\lceil 2G_\vareps\left(K\sqrt{\frac{L_*}{\mu}}+\sqrt{\frac{H}{\mu}}\frac{\sqrt{C_\beta(\sigma-1)}}{\sqrt{\vareps}(\sqrt{\sigma}-1)}\right)+K\right\rceil, \text{ if }\mu>0.
\end{equation}
where 
$$G_\vareps = \log\frac{C_\beta D^2}{\vareps C_\vareps}+\log\left(L_*+\mu+\frac{H\big(C_\beta(\sigma-1)+\beta_g\vareps\big)}{\sigma\vareps}\right).$$
\end{theorem}

\begin{proof}
When $\mu=0$, we have from the first inequality in \eqref{eq:tk-cvx} that the total number of gradient evaluations satisfies
\begin{equation}
T_K\le \sum_{k=0}^{K-1}\frac{\dist(\vx^k,\cX_k^*)\sqrt{2(L_*+\beta_k H)}}{\sqrt{\vareps_k}}+K.\label{eq:total-grad-cvx-adp2}
\end{equation}
Plugging into \eqref{eq:total-grad-cvx-adp2} the $\vareps_k$ given in \eqref{eq:const-eps} and noting $\dist(\vx^k ,\cX_k^*)\le D$ yields
\begin{equation}\label{eq:TK-geo-pen-const-eps}
T_K\le 2D\sqrt{\frac{C_\beta}{C_\vareps}}\sum_{k=0}^{K-1} \frac{\sqrt{L_*+\beta_k H}}{\sqrt{\vareps}}+K.
\end{equation}
Note that
$\sum_{k=0}^{K-1}\sqrt{\beta_k}=\sqrt{\beta_g}\frac{\sigma^{\frac{K}{2}}-1}{\sqrt{\sigma}-1}.$
From \eqref{eq:def-beta-g}, it holds
\begin{equation}\label{eq:sigmaK}
\sigma^K=\frac{C_\beta(\sigma-1)}{\beta_g\vareps}+1,
\end{equation}
and thus
$\sigma^{\frac{K}{2}}-1\le \sqrt{\frac{C_\beta(\sigma-1)}{\beta_g\vareps}}.$
Therefore, 
\begin{equation}\label{eq:bd-sum-beta-sqrt}
\sum_{k=0}^{K-1}\sqrt{\beta_k}\le \frac{\sqrt{C_\beta (\sigma-1)}}{\sqrt{\vareps}(\sqrt{\sigma}-1)},
\end{equation}
and using $\sqrt{L_*+\beta_k H}\le \sqrt{L_*}+\sqrt{\beta_k H}$, we have
\begin{equation}\label{eq:bd-Lstar-betaH}
\sum_{k=0}^{K-1}\sqrt{L_*+\beta_k H}\le \sum_{k=0}^{K-1}\left(\sqrt{L_*}+\sqrt{\beta_k H}\right)\le K\sqrt{L_*}+\frac{\sqrt{C_\beta H(\sigma-1)}}{\sqrt{\vareps}(\sqrt{\sigma}-1)},
\end{equation}
which together with \eqref{eq:TK-geo-pen-const-eps} gives \eqref{eq:T-ialm-cvx-const-eps}.

For the strongly convex case, we use \eqref{eq:bd-denom} and the second inequality of \eqref{eq:tk-cvx} to have
\begin{equation}
T_K\le 2\sum_{k=0}^{K-1}\sqrt{\frac{L_*+\beta_k H}{\mu}}\log\left(\frac{L_*+\beta_k H+\mu}{2\vareps_k}[\dist(\vx^k,\cX_k^*)]^2\right) + K.\label{eq:total-grad-scvx-adp2}
\end{equation}
Since $\dist(\vx^k,\cX_k^*)\le D$ and $\vareps_k$'s are set to those in \eqref{eq:const-eps}, the above inequality indicates
\begin{equation}\label{eq:TK-geo-pen-const-eps-scvx}T_K\le 2\sum_{k=0}^{K-1}\sqrt{\frac{L_*+\beta_k H}{\mu}}\log\frac{C_\beta D^2(L_*+\beta_k H+\mu)}{\vareps C_\vareps} + K.
\end{equation}
For $0\le k<K$,
\begin{equation}\label{eq:est-beta-K-1}
\beta_k\le \beta_{K-1}=\beta_g\sigma^{K-1}=\frac{\beta_g}{\sigma}\sigma^K\overset{\eqref{eq:sigmaK}}=\frac{\beta_g}{\sigma}\left(\frac{C_\beta(\sigma-1)}{\beta_g\vareps}+1\right)=\frac{C_\beta(\sigma-1)+\beta_g\vareps}{\sigma\vareps}.
\end{equation}
Plugging into \eqref{eq:TK-geo-pen-const-eps-scvx} the second inequality in \eqref{eq:bd-Lstar-betaH} and the above bound on $\beta_k$, we have \eqref{eq:T-ialm-strcvx-const-eps} and thus complete the proof. 
\end{proof}

\begin{remark}
Comparing the iteration complexity results in Theorems \ref{thm:it-comp-const} and \ref{thm:iter-comp-geo-const-eps}, we see that if $K=1$, the number $T_K$ in either case of $\mu=0$ or $\mu>0$ is the same for both penalty parameter settings as $\sigma\to\infty$. That is because when $K=1$, iALM with either of the two settings reduces to the penalty method. If $K>1$, the number $T_K$ for the setting of geometrically increasing penalty can be smaller than that for the constant parameter setting as $\sigma$ is big; see numerical results in section \ref{sec:numerical}.
\end{remark}

\begin{theorem}[Iteration complexity with geometrically increasing penalty and adaptive error]\label{thm:iter-comp-geo}
For any given $\vareps>0$, let $K$ be a positive integer number and $C_\beta, C_\vareps$ two positive real numbers. Set $\{\beta_k\}$ and $\{\rho_k\}$ according to Setting \ref{set-ialm}.  Assume $\mu\le \frac{L_0}{4}$. If $\mu=0$, set $\{\vareps_k\}$ as in \eqref{eq:best-eps-cvx-app}, and if $\mu>0$, set $\{\vareps_k\}$ as in \eqref{eq:best-eps-scvx-app}. Let $\bar{\vx}^K$ be given in \eqref{eq:barx-K}. Then the inequalities in \eqref{eq:error-barx-K} hold, and Algorithm \ref{alg:ialm} can produce $\bar{\vx}^K$ by evaluating gradients of $g, f_i, \,i\in [m]$ in at most $T_K$ times, where  
\begin{equation}\label{eq:T-ialm-cvx}
T_K\le \left\lceil 2D\sqrt{\frac{C_\beta}{C_\vareps}}\left(\sqrt{\frac{L_*}{\vareps}}\frac{(\sigma-1)^{\frac{1}{2}}}{(\sigma^{\frac{1}{6}}-1)(\sigma^{\frac{2}{3}}-1)^{\frac{1}{2}}}+\frac{\sqrt{HC_\beta}(\sigma-1)}{\vareps(\sigma^{\frac{2}{3}}-1)^{\frac{3}{2}}}\right)+K\right\rceil, \text{ if }\mu=0,
\end{equation}
and
\begin{equation}\label{eq:T-ialm-strcvx}
T_K\le \left\lceil 2G_\vareps\left(K\sqrt{\frac{L_*}{\mu}}+\sqrt{\frac{H}{\mu}}\frac{\sqrt{C_\beta(\sigma-1)}}{\sqrt{\vareps}(\sqrt{\sigma}-1)}\right)+K\right\rceil, \text{ if }\mu>0.
\end{equation}
where
$$G_\vareps=\log\frac{C_\beta D^2}{\vareps C_\vareps}+\log\left(L_*+\mu+\frac{H\big(C_\beta(\sigma-1)+\beta_g\vareps\big)}{\sigma\vareps}\right)+\log\frac{\sqrt{(\sigma-1)^2+\beta_g\vareps(\sigma-1)/C_\beta}}{\sigma-\sqrt{\sigma}}.$$

\end{theorem}
\begin{proof}
For the case of $\mu=0$, we have \eqref{eq:total-grad-cvx-adp2}, plugging into which the $\vareps_k$ given in \eqref{eq:best-eps-cvx-app} yields
$$T_K\le \frac{2}{\sqrt{C_\vareps}}\sqrt{\sum_{t=0}^{K-1}\beta_t^{\frac{2}{3}}}\sum_{k=0}^{K-1}\dist(\vx^k,\cX_k^*)\beta_k^{\frac{1}{6}}(L_*+\beta_k H)^{\frac{1}{2}}+K.$$
Since $\dist(\vx^k ,\cX_k^*)\le D$, the above inequality implies
\begin{align}\label{eq:bd-TK-ialm-cvx}
T_K \le  \frac{2D}{\sqrt{C_\vareps}}\sqrt{\sum_{t=0}^{K-1}\beta_t^{\frac{2}{3}}}\sum_{k=0}^{K-1}\beta_k^{\frac{1}{6}}(L_*+\beta_k H)^{\frac{1}{2}}+K.
\end{align}
Note that
$$\sqrt{\sum_{t=0}^{K-1}\beta_t^{\frac{2}{3}}}=\sqrt{\sum_{t=0}^{K-1}\beta_g^{\frac{2}{3}}\sigma^{\frac{2t}{3}}}=\beta_g^{\frac{1}{3}}\sqrt{\frac{\sigma^{\frac{2K}{3}}-1}{\sigma^{\frac{2}{3}}-1}},$$
and
\begin{align*}\sum_{k=0}^{K-1}\beta_k^{\frac{1}{6}}(L_*+\beta_k H)^{\frac{1}{2}}\le \sum_{k=0}^{K-1}\beta_k^{\frac{1}{6}}\left(\sqrt{L_*}+\sqrt{\beta_k H}\right)=&~\sqrt{L_*}\sum_{k=0}^{K-1}\beta_k^{\frac{1}{6}}+\sqrt{H}\sum_{k=0}^{K-1}\beta_k^{\frac{2}{3}}\\
=&~\sqrt{L_*}\beta_g^{\frac{1}{6}}\frac{\sigma^{\frac{K}{6}}-1}{\sigma^{\frac{1}{6}}-1}+\sqrt{H}\beta_g^{\frac{2}{3}}\frac{\sigma^{\frac{2K}{3}}-1}{\sigma^{\frac{2}{3}}-1}.
\end{align*}
Hence, it follows from \eqref{eq:bd-TK-ialm-cvx} that
\begin{equation}\label{eq:bd-sum-32}T_K\le \frac{2D}{\sqrt{C_\vareps}}\beta_g^{\frac{1}{3}}\sqrt{\frac{\sigma^{\frac{2K}{3}}-1}{\sigma^{\frac{2}{3}}-1}}\left(\sqrt{L_*}\beta_g^{\frac{1}{6}}\frac{\sigma^{\frac{K}{6}}-1}{\sigma^{\frac{1}{6}}-1}+\sqrt{H}\beta_g^{\frac{2}{3}}\frac{\sigma^{\frac{2K}{3}}-1}{\sigma^{\frac{2}{3}}-1}\right)+K.
\end{equation}
From \eqref{eq:sigmaK} and the fact $\sqrt{a+b}\le \sqrt{a}+\sqrt{b},\,\forall a,b\ge 0$, it follows that
\begin{equation}\label{eq:bd-sig-K4}
\sigma^{\frac{2K}{3}}-1\le \left(\frac{C_\beta(\sigma-1)}{\beta_g\vareps}\right)^{\frac{2}{3}},\quad \sigma^{\frac{K}{6}}-1\le \left(\frac{C_\beta(\sigma-1)}{\beta_g\vareps}\right)^{\frac{1}{6}}.
\end{equation} 
Therefore, plugging the two inequalities in \eqref{eq:bd-sig-K4} into \eqref{eq:bd-sum-32} yields \eqref{eq:T-ialm-cvx}.


For the case of $\mu>0$, we have \eqref{eq:total-grad-scvx-adp2}. Since $\dist(\vx^k,\cX_k^*)\le D$ and $\vareps_k$'s are set to those in \eqref{eq:best-eps-scvx-app}, the inequality in \eqref{eq:total-grad-scvx-adp2} indicates
\begin{align}\label{eq:total-grad-scvx-adp2-D}
T_K\le &~2\sum_{k=0}^{K-1}\sqrt{\frac{L_*+\beta_k H}{\mu}}\log\left(\sqrt{\beta_k}\sum_{t=0}^{K-1}\sqrt{\beta_t}\frac{L_*+\beta_k H+\mu}{C_\vareps}D^2\right) + K\cr
=&~2\sum_{k=0}^{K-1}\sqrt{\frac{L_*+\beta_k H}{\mu}}\left(\log\frac{\sqrt{\beta_k}D^2\sum_{t=0}^{K-1}\sqrt{\beta_t}}{C_\vareps}+\log\big(L_*+\beta_k H+\mu\big)\right)+K.
\end{align}
Therefore, plugging into \eqref{eq:total-grad-scvx-adp2-D} the inequality in \eqref{eq:bd-sum-beta-sqrt}, the upper bounds of $\sum_{k=0}^{K-1}\sqrt{L_*+\beta_k H}$ and $\beta_k$ in \eqref{eq:bd-Lstar-betaH} and \eqref{eq:est-beta-K-1} respectively, we obtain \eqref{eq:T-ialm-strcvx} and complete the proof.  
\end{proof}

\begin{remark}
Let us compare the iteration complexity results in Theorems \ref{thm:iter-comp-geo-const-eps} and \ref{thm:iter-comp-geo}. We see that for the case of $\mu=0$, as $K>1$ and $\sigma$ is big, if $\sqrt{\frac{L_*}{\vareps}}$ dominates $\frac{\sqrt{H C_\beta}}{\vareps}$, the iteration complexity result in Theorem \ref{thm:iter-comp-geo} is better than that in Theorem \ref{thm:iter-comp-geo-const-eps} (see the numerical results in Table \ref{table:cvx-big-Lip}), and if $\frac{\sqrt{H C_\beta}}{\vareps}$ dominates $\sqrt{\frac{L_*}{\vareps}}$, the two results are similar. For the case of $\mu>0$, as $K>1$, the iteration complexity result in Theorem \ref{thm:iter-comp-geo-const-eps} is better than that in Theorem \ref{thm:iter-comp-geo}.
\end{remark}

\section{Nonergodic convergence rate and iteration complexity}\label{sec:nonergo-rate}
In this section, we show a nonergodic convergence rate result of Algorithm \ref{alg:ialm}, by employing the relation between iALM and the inexact proximal point algorithm (iPPA). 
Throughout this section, we assume there is no affine equality constraint in \eqref{eq:ccp}, i.e., we consider the problem
\begin{equation}\label{eq:ineq-ccp}
\Min_{\vx\in \cX} f_0(\vx), \st f_i(\vx)\le 0, \, \forall i\in [m],
\end{equation}
where $f_i, i=0,1,\ldots,m,$ satisfy the assumptions through \eqref{eq:f0}--\eqref{eq:lip-f}.
We do not include affine equality constraints for the purpose of directly applying existing results in \cite{rockafellar1973dual, rockafellar1976augmented}. Although results similar to those in \cite{rockafellar1973dual, rockafellar1976augmented} can possibly be shown for the equality and inequality constrained problem \eqref{eq:ccp}, we do not extend our discussion but instead formulate any affine equality constraint $\va^\top\vx=b$ by two affine inequality constraints $\va^\top\vx-b\le0$ and $-\va^\top\vx+b\le0$ if there is any.

\subsection{Relation between iALM and iPPA}\label{sec:rel-ialm-ippa}
Let $\cL_0(\vx,\vz)$ be the Lagrangian function of \eqref{eq:ineq-ccp}, namely,
$$\cL_0(\vx,\vz)=f_0(\vx) + \sum_{i=1}^mz_i f_i(\vx),$$
and let $\cL_\beta(\vx,\vz)$ be the augmented Lagrangian function of \eqref{eq:ineq-ccp}, defined in the same way as that in \eqref{eq:aug-fun}.
In addition, let $d_0(\vz)$ be the Lagrangian dual function, defined as
$$d_0(\vz)=\left\{\begin{array}{ll}\min_{\vx\in\cX} \cL_0(\vx,\vz), & \text{ if }\vz\ge \vzero,\\[0.1cm]
-\infty, & \text{ otherwise},
\end{array}
\right.$$
and let $d_\beta(\vz)\triangleq \min_{\vx\in\cX} \cL_\beta(\vx,\vz)$ be the augmented Lagrangian dual function. 

Applying Algorithm \ref{alg:ialm} with $\rho_k=\beta_k$ to \eqref{eq:ineq-ccp}, we have iterates $\{(\vx^k,\vz^k)\}$ that satisfy:
\begin{subequations}\label{eq:ineq-ialm}
\begin{align}
&\cL_{\beta_k}(\vx^{k+1},\vz^k) \le d_\beta(\vz^k) + \vareps_k,\label{eq:ineq-ialm-x}\\[0.1cm]
&\vz^{k+1}=\vz^k+\beta_k\nabla_\vz\cL_{\beta_k}(\vx^{k+1},\vz^k).\label{eq:ineq-ialm-z}
\end{align}
\end{subequations}
The iPPA applied to the Lagrangian dual problem $\max_\vz d_0(\vz)$ iteratively performs the updates:
\begin{equation}\label{eq:ippa-d0}
\vz^{k+1}\approx \cM_{\beta_k}(\vz^k),
\end{equation}
where the operator $\cM_{\beta}$ is the proximal mapping of $-\beta d_0$, defined as
$$\cM_{\beta}(\vz)=\argmax_\vu d_0(\vu)-\frac{1}{2\beta}\|\vu-\vz\|^2.$$
In \eqref{eq:ippa-d0}, the approximation could be measured by the objective error as in \eqref{eq:ineq-ialm-x} or by the gradient norm at the returned point $\vz^{k+1}$; see \cite{guler1992new} for example.

It was noted in \cite{rockafellar1973dual} that
\begin{equation}\label{eq:rel-alm-ppa}
d_\beta(\vz) = \max_{\vu} d_0(\vu)-\frac{1}{2\beta}\|\vu-\vz\|^2,
\end{equation} 
and in addition, if $\hat{\vx}\in\cX$ satisfies $\cL_\beta(\hat{\vx}, \vz) \le d_\beta(\vz)+\vareps$, then
$\|\vz + \beta \nabla_\vz \cL_\beta(\hat{\vx}, \vz) - \cM(\vz)\|\le \sqrt{\frac{\beta\vareps}{2}}.$ 
Therefore, iALM with updates in \eqref{eq:ineq-ialm} reduces to iPPA in \eqref{eq:ippa-d0} with approximation error 
\begin{equation}\label{eq:diff-zk-true-z}
\|\vz^{k+1}-\cM_{\beta_k}(\vz^k)\|\le \sqrt{\frac{\beta_k\vareps_k}{2}}.
\end{equation}

\subsection{Nonergodic convergence rate of iALM}
For iALM with updates in \eqref{eq:ineq-ialm} on solving \eqref{eq:ineq-ccp}, \cite[Theorem 4]{rockafellar1976augmented} establishes the following bounds on the objective error and feasibility violation:
\begin{subequations}\label{eq:bd-rockafellar}
\begin{align}
&f_0(\vx^{k+1})-f_0(\vx^*)\le \vareps_k+\frac{\|\vz^k\|^2-\|\vz^{k+1}\|^2}{2\beta_k},\label{eq:bd-rockafellar-obj}\\[0.1cm]
&f_i(\vx^{k+1})\le \frac{|z_i^k-z_i^{k+1}|}{\beta_k},\, \forall \, i\in[m].\label{eq:bd-rockafellar-res}
\end{align}
\end{subequations}
If in \eqref{eq:ineq-ialm-x}, $\vareps_k=0,\,\forall k$, \cite[Theorem 2.2]{guler1991convergence} shows that
\begin{equation}\label{eq:rate-ppa}
\frac{\|\vz^k-\vz^{k+1}\|}{\beta_k}\le \frac{\|\vz^0-\vz^*\|}{\sum_{t=0}^k\beta_t}.
\end{equation}
Therefore, combining the results in \eqref{eq:bd-rockafellar} with $\vareps_k=0,\,\forall k$ and \eqref{eq:rate-ppa}, and also noting the boundedness of $\vz^k$ from \eqref{eq:bd-zk}, one can easily obtain a nonergodic convergence rate result of exact ALM on solving \eqref{eq:ineq-ccp}. However, if $\vareps_k>0$, we do not notice any existing result on estimating $\frac{\|\vz^k-\vz^{k+1}\|}{\beta_k}$. In the following, we establish a bound on this quantity and thus show a nonergodic convergence rate result of iALM. 

\begin{lemma}\label{lem:bd-diff-zk-zk+1}
Given a positive integer $K$ and a nonnegative number $C_\vareps$, choose positive sequences $\{\beta_k\}$ and $\{\vareps_k\}$ such that $\sum_{k=0}^{K-1}\beta_k\vareps_k\le \frac{C_\vareps}{2}$. Let $\{(\vx^k,\vz^k)\}_{k=0}^{K}$ be the sequence generated from the updates in \eqref{eq:ineq-ialm} with $\vz^0=\vzero$ on solving \eqref{eq:ineq-ccp}. Then
\begin{equation}\label{bd-diff-zk-zk+1}
\|\vz^k-\vz^{k+1}\|\le 5\|\vz^*\|+\frac{7}{2}\sqrt{C_\vareps},
\end{equation}
where we have assumed that \eqref{eq:ineq-ccp} has a primal-dual solution $(\vx^*,\vz^*)$. 
\end{lemma}
\begin{proof}
Let $\tilde{\vz}^{k+1}=\cM_{\beta_k}(\vz^k)$. Then from \eqref{eq:rel-alm-ppa}, it follows that
\begin{equation}\label{eq:diff-zk-ztilde}
\frac{1}{2\beta_k}\|\tilde{\vz}^{k+1}-\vz^k\|^2=d_0(\tilde{\vz}^{k+1})-d_{\beta_k}(\vz^k).
\end{equation}
By the weak duality, it holds $d_0(\tilde{\vz}^{k+1})\le f_0(\vx^*)$. From \eqref{eq:ineq-ialm-x}, we have
\begin{equation}\label{eq:dcLk}
d_{\beta_k}(\vz^k)\ge \cL_{\beta_k}(\vx^{k+1},\vz^k)-\vareps_k.
\end{equation}
Recall the definition of $\psi_\beta(u,v)$ in \eqref{eq:def-psi} and note $\psi_\beta(u,v)\ge-\frac{v^2}{2\beta}$. Hence,
$$\cL_{\beta_k}(\vx^{k+1},\vz^k)=f_0(\vx^{k+1})+\sum_{i=1}^m\psi_{\beta_k}(f_i(\vx^{k+1},z_i^k)\ge f_0(\vx^{k+1})-\frac{\|\vz^k\|^2}{2\beta_k}.$$
Thus by \eqref{eq:dcLk}, it holds that
$$d_{\beta_k}(\vz^k)\ge f_0(\vx^{k+1})-\frac{\|\vz^k\|^2}{2\beta_k}-\vareps_k,$$
and
\begin{equation}\label{eq:d0-diff-dbeta}
d_0(\tilde{\vz}^{k+1})-d_{\beta_k}(\vz^k)\le f_0(\vx^*)-f_0(\vx^{k+1})+\frac{\|\vz^k\|^2}{2\beta_k}+\vareps_k.
\end{equation}

Since $(\vx^*,\vz^*)$ is a primal-dual solution of \eqref{eq:ineq-ccp}, similar to \eqref{eq:opt}, it holds that 
\begin{equation}\label{eq:opt-ineq}
f_0(\vx)-f_0(\vx^*)+\sum_{i=1}^m z_i^* f_i(\vx)\ge0,\,\forall \vx\in \cX.
\end{equation}
Hence,
$$f_0(\vx^*)-f_0(\vx^{k+1})\le \sum_{i=1}^m z_i^* f_i(\vx^{k+1})\overset{\eqref{eq:bd-rockafellar-res}}\le \frac{\|\vz^k-\vz^{k+1}\|}{\beta_k}\|\vz^*\|.$$
Therefore, from \eqref{eq:d0-diff-dbeta} and the above inequality, it follows that
$$d_0(\tilde{\vz}^{k+1})-d_{\beta_k}(\vz^k)\le \frac{\|\vz^k-\vz^{k+1}\|}{\beta_k}\|\vz^*\|+\frac{\|\vz^k\|^2}{2\beta_k}+\vareps_k,$$
and noting \eqref{eq:diff-zk-ztilde}, we have
\begin{align*}
\|\tilde{\vz}^{k+1}-\vz^k\|= &~ \sqrt{2\beta_k\big(d_0(\tilde{\vz}^{k+1})-d_{\beta_k}(\vz^k)\big)}\\
\le &~\sqrt{2\|\vz^k-\vz^{k+1}\|\cdot\|\vz^*\|+\|\vz^k\|^2+2\beta_k\vareps_k}.
\end{align*}
By the triangle inequality, we have from \eqref{eq:diff-zk-true-z} and the above inequality that
\begin{align}\label{eq:bd-diff-2zk}
\|\vz^k-\vz^{k+1}\|\le  &~\sqrt{2\|\vz^k-\vz^{k+1}\|\cdot\|\vz^*\|+\|\vz^k\|^2+2\beta_k\vareps_k}+\sqrt{\frac{\beta_k\vareps_k}{2}}\cr
\le & \sqrt{2(\|\vz^k\|+\|\vz^{k+1}\|)\|\vz^*\|} + \|\vz^k\|+\frac{3\sqrt{2}}{2}\sqrt{\beta_k\vareps_k}\cr
\le & \frac{\|\vz^k\|+\|\vz^{k+1}\|}{2}+\|\vz^*\|+\|\vz^k\|+\frac{3}{2}\sqrt{C_\vareps},
\end{align}
where the last inequality uses the Young's inequality and the fact $\beta_k\vareps_k\le \frac{C_\vareps}{2}$.
Through the same arguments as those in Lemma \ref{lem:bd-z-Lz}, one can show
\begin{equation}\label{eq:bd-zk-ineq-case}
\|\vz^k\|\le 2\|\vz^*\|+\sqrt{C_\vareps},\,\,\forall \, 0\le k\le K,
\end{equation}
i.e., the inequality in \eqref{eq:bd-zk} with $\vy^*=\vzero$.
Plugging the above bound on $\|\vz^k\|$ into \eqref{eq:bd-diff-2zk} gives the desired result in \eqref{bd-diff-zk-zk+1}.
\end{proof}

Combining \eqref{eq:bd-rockafellar} and \eqref{bd-diff-zk-zk+1}, we are able to establish the nonergodic convergence rate result of iALM on solving \eqref{eq:ineq-ccp}.

\begin{theorem}[nonergodic convergence rate]\label{thm:nonerg-rate}
Under the same assumptions of Lemma \ref{lem:bd-diff-zk-zk+1}, it holds that for any $0\le k < K$,
\begin{subequations}\label{eq:nonerg-rate}
\begin{align}
&\big|f_0(\vx^{k+1})-f_0(\vx^*)\big|\le \vareps_k+\frac{2\|\vz^*\|+\sqrt{C_\vareps}}{\beta_k}\left(5\|\vz^*\|+\frac{7}{2}\sqrt{C_\vareps}\right),\label{eq:nonerg-rate-obj}\\[0.1cm]
&\big\|[\vf(\vx^{k+1})]_+\big\|\le \frac{1}{\beta_k}\left(5\|\vz^*\|+\frac{7}{2}\sqrt{C_\vareps}\right).\label{eq:nonerg-rate-res}
\end{align}
\end{subequations}

\end{theorem}

\begin{proof}
Directly from \eqref{eq:bd-rockafellar},  \eqref{bd-diff-zk-zk+1}, and \eqref{eq:bd-zk-ineq-case}, we obtain \eqref{eq:nonerg-rate-res} and
\begin{equation}\label{eq:obj-error-right}
f_0(\vx^{k+1})-f_0(\vx^*)\le \vareps_k+\frac{2\|\vz^*\|+\sqrt{C_\vareps}}{\beta_k}\left(5\|\vz^*\|+\frac{7}{2}\sqrt{C_\vareps}\right).
\end{equation}
Using \eqref{eq:opt-ineq}, we have from \eqref{eq:nonerg-rate-res} that
$$f_0(\vx^{k+1})-f_0(\vx^*)\ge -\frac{\|\vz^*\|}{\beta_k}\left(5\|\vz^*\|+\frac{7}{2}\sqrt{C_\vareps}\right),$$
which together with \eqref{eq:obj-error-right} gives \eqref{eq:nonerg-rate-obj}.
\end{proof}

\begin{remark}
From the results in \eqref{eq:nonerg-rate}, we see that to have $\{\vx^k\}$ to be a minimizing sequence of \eqref{eq:ineq-ccp}, we need $\beta_k\to\infty$ and $\vareps_k\to 0$ as $k\to \infty$. Hence, setting $\{\beta_k\}$ to a constant sequence will not be a valid option. 
\end{remark}

\subsection{Iteration complexity}
In this subsection, we set parameters according to Setting \ref{set-ialm}, and we estimate the iteration complexity of iALM on solving \eqref{eq:ineq-ccp} by applying Nesterov's optimal first-order method to \eqref{eq:ineq-ialm-x}. Again, note that the results in Theorem \ref{thm:nonerg-rate} do not need specific structure of \eqref{eq:ineq-ccp} except convexity. Hence, if the problem has richer structures, one can apply more efficient methods to find $\vx^{k+1}$ that satisfies \eqref{eq:ineq-ialm-x}. 

\begin{theorem}[Nonergodic iteration complexity]\label{thm:nonerg-iter}
Given a positive integer $K$ and positive numbers $C_\beta, C_\vareps$, choose positive sequences $\{\rho_k\}$ and $\{\beta_k\}$ according to Setting \ref{set-ialm}. In addition, choose $\{\vareps_k\}$ according to \eqref{eq:const-eps} for both cases of $\mu=0$ and $\mu>0$, or choose $\{\vareps_k\}$ according to \eqref{eq:best-eps-cvx-app} for the case of $\mu=0$ and \eqref{eq:best-eps-scvx-app} for $\mu>0$. Let $\{(\vx^k,\vz^k)\}_{k=0}^{K}$ be the sequence generated from Algorithm \ref{alg:ialm} with $\vy^k=\vzero,\,\forall k$, and $\vz^0=\vzero$ on solving \eqref{eq:ineq-ccp}. Then
\begin{subequations}\label{eq:nonerg-error-bd}
\begin{align}
&\big|f_0(\vx^{K})-f_0(\vx^*)\big|\le \frac{\vareps}{2}\frac{C_\vareps}{C_\beta}+\frac{\vareps\sigma}{C_\beta(\sigma-1)}\big(2\|\vz^*\|+\sqrt{C_\vareps}\big)\left(5\|\vz^*\|+\frac{7}{2}\sqrt{C_\vareps}\right),\label{eq:nonerg-error-bd-obj}\\[0.1cm]
&\big\|[\vf(\vx^{K})]_+\big\|\le \frac{\vareps\sigma}{C_\beta(\sigma-1)}\left(5\|\vz^*\|+\frac{7}{2}\sqrt{C_\vareps}\right).\label{eq:nonerg-error-bd-res}
\end{align}
\end{subequations}
If $\{\vareps_k\}$ is chosen according to \eqref{eq:const-eps} for both cases of $\mu=0$ and $\mu>0$, the total number $T_K$ of gradient evaluations is given in \eqref{eq:T-ialm-cvx-const-eps} and \eqref{eq:T-ialm-strcvx-const-eps} respectively; if $\{\vareps_k\}$ is set according to \eqref{eq:best-eps-cvx-app} for the case of $\mu=0$ and \eqref{eq:best-eps-scvx-app} for $\mu>0$, then $T_K$ is given in \eqref{eq:T-ialm-cvx} for $\mu=0$ and \eqref{eq:T-ialm-strcvx} for $\mu>0$. 
\end{theorem}

\begin{proof}
Note that $\beta_k$ is increasing with respect to $k$. Hence, the $\vareps_k$ given in both \eqref{eq:best-eps-cvx-app} and \eqref{eq:best-eps-scvx-app} is decreasing, and thus 
$$\vareps_{K-1}\le \frac{\sum_{t=0}^{K-1}\beta_t\vareps_t}{\sum_{t=0}^{K-1}\beta_t}\le \frac{\vareps}{2}\frac{C_\vareps}{C_\beta}.$$
If $\{\vareps_k\}$ is chosen according to \eqref{eq:const-eps} for both cases of $\mu=0$ and $\mu>0$, then the above bound on $\vareps_{K-1}$ obviously holds.
In addition, from \eqref{eq:est-beta-K-1}, we have 
$$\beta_{K-1}\ge \frac{C_\beta(\sigma-1)}{\vareps\sigma}.$$
Therefore, plugging into \eqref{eq:nonerg-rate} the bounds on $\vareps_{K-1}$ and $\beta_{K-1}$ gives the desired results in \eqref{eq:nonerg-error-bd}.

The bounds on the total number $T_K$ of gradient evaluations follow from the same arguments as in the proofs of Theorems \ref{thm:iter-comp-geo-const-eps} and \ref{thm:iter-comp-geo}. Hence, we complete the proof.
\end{proof}

\begin{remark}
From the results in \eqref{eq:nonerg-error-bd}, we see that if 
\begin{equation}\label{eq:bd-C-beta2}
2C_\beta\ge C_\vareps + \frac{2\sigma}{(\sigma-1)}\big(2\|\vz^*\|+\sqrt{C_\vareps}\big)\left(5\|\vz^*\|+\frac{7}{2}\sqrt{C_\vareps}\right),
\end{equation}
then $\vx^K$ is an $\vareps$-optimal solution to \eqref{eq:ineq-ccp}. If $\|\vz^*\|\ge 1$, $C_\vareps = \|\vz^*\|^2$, and $\frac{\sigma}{\sigma-1}\approx 1$ (e.g., $\sigma=10$ is often used), then the $C_\beta$ in \eqref{eq:bd-C-beta2} is roughly 10 times of that in \eqref{eq:bd-C-beta1} by assuming no affine constraint. For the iteration complexity, if $\sqrt{\frac{L_*}{\vareps}}$ dominates $\frac{\sqrt{H}\|\vz^*\|}{\vareps}$, then the nonergodic result is roughly $\sqrt{10}$ times of the ergodic result for both convex and strongly convex cases. If $\frac{\sqrt{H}\|\vz^*\|}{\vareps}$ dominates, then the former would be roughly 10 times of the latter for the convex case, but still roughly $\sqrt{10}$ times for the strongly convex case. However, in either case, both ergodic and nonergodic results have the same order of complexity. 
\end{remark}

\section{Related works and comparison with existing results}\label{sec:review}
In this section, we review related works and compare them to our results. Our review and comparison focus on convex optimization, but note that ALM has also been popularly applied to nonconvex optimization problems; see \cite{bertsekas2014constrained, bertsekas1999nonlinear, birgin2005numerical-alm} and the references therein. 

\subsection*{Affinely constrained convex problems} 
Several recent works have established the convergence rate of ALM and its inexact version for affinely constrained convex problems:
\begin{equation}\label{eq:lin-cp}\Min_{\vx\in\cX} f_0(\vx), \st \vA\vx=\vb.
\end{equation} Assuming exact solution to every $\vx$-subproblem, \cite{he2010aalm} first shows $O(1/k)$ convergence of ALM for smooth problems in terms of dual objective and then accelerates the rate to $O(1/k^2)$ by applying Nesterov's extrapolation technique to the multiplier update. The results are extended to nonsmooth problems in \cite{kang2013accelerated} that uses similar technique. By adapting parameters, \cite{xu2017accelerated} establishes $O(1/k^2)$ convergence of a linearized ALM in terms of primal objective and feasibility violation. The linearized ALM allows linearization to smooth part in the objective but still assumes exact solvability of $\vx$-subproblems. 

When the objective is strongly convex, \cite{kang2015inexact} proves $O(1/k^2)$ convergence of iALM with extrapolation technique applied to the multiplier update. It requires summable error and subproblems to be solved more and more accurately. However, it does not give an estimate on the total number of gradient evaluations on solving all subproblems to the required accuracies. 

For smooth linearly constrained convex problems, \cite{lan2016iteration-alm} analyzes the iteration complexity of the iALM. It applies Nesterov's optimal first-order method to every $\vx$-subproblem and shows that $O(\vareps^{-\frac{7}{4}})$ gradient evaluations are  required to reach an $\vareps$-optimal solution. Compared to this complexity, our results for the convex case are better by an order $O(\vareps^{-\frac{3}{4}})$. In addition, \cite{lan2016iteration-alm} modifies the iALM by solving a perturbed problem. The modified iALM requires $O(\vareps^{-1}|\log \vareps |^{\frac{3}{4}})$ gradient evaluations to produce an $\vareps$-optimal solution, and this order is worse than our results by an order $O(|\log \vareps |^{\frac{3}{4}})$. Motivated by the model predictive control, \cite{nedelcu2014computational} also analyzes the iteration complexity of inexact dual gradient methods (iDGM) that are essentially iALMs. It shows that to reach an $\vareps$-optimal solution\footnote{\cite{nedelcu2014computational} assumes every subproblem solved to the condition $\langle \tilde{\nabla} \cL_\beta(\vx^{k+1},\vy^k ), \vx-\vx^{k+1}\rangle \ge -O(\vareps),\,\forall \vx\in\cX$, which is implied by $\cL_\beta(\vx^{k+1},\vy^k )-\min_{\vx\in\cX}\cL_\beta(\vx,\vy^k )\le O(\vareps^2)$ if $\cL_\beta$ is Lipschitz differentiable with respect to $\vx$.}, a nonaccelerated iDGM requires $O(\vareps^{-1})$ outer iterations and every $\vx$-subproblem solved to an accuracy $O(\vareps^2)$, and an accelerated iDGM requires $O(\vareps^{-\frac{1}{2}})$ outer iterations and every $\vx$-subproblem solved to an accuracy $O(\vareps^3)$. While the iteration complexity in \cite{lan2016iteration-alm} is estimated based on the best iterate, and that in \cite{nedelcu2014computational} is ergodic, the recent work \cite{liu2016iALM} establishes non-ergodic convergence of iALM. It requires $O(\vareps^{-2})$ gradient evaluations to reach an $\vareps$-optimal primal-dual solution $(\bar{\vx},\bar{\vy})$ in the sense that 
\begin{equation}\label{eq:vareps-liu}
\|\vA\bar{\vx}-\vb\|\le \sqrt{\vareps},\quad \left\langle \nabla g(\bar{\vx})+\vA^\top \bar{\vy}, \bar{\vx}-\vx\right\rangle+h(\bar{\vx})-h(\vx) \le \vareps,\,\forall \vx,
\end{equation} 
where it is assumed that $f_0=g+h$ in \eqref{eq:lin-cp} and $g$ is Lipschitz differentiable. From the convexity of $g$, it follows from \eqref{eq:vareps-liu} that
$$f_0(\bar{\vx}) - f_0(\vx^*) \le \left\langle \nabla g(\bar{\vx}), \bar{\vx}-\vx^*\right\rangle+h(\bar{\vx})-h(\vx)\le \vareps - \langle \bar{\vy}, \vA\bar{\vx}-\vb\rangle\le \vareps+\|\bar{\vy}\|\sqrt{\vareps}.$$ 
Hence, if $\bar{\vy}\neq\vzero$, to have an $\vareps$-optimal solution by our Definition \ref{def:eps-opt}, the iteration complexity result in  
\cite{liu2016iALM} would be $O(\vareps^{-4})$, which is $O(\vareps^{-3})$ worse than our nonergodic iteration complexity result in Theorem \ref{thm:nonerg-iter}.

Another line of existing works on iALM assume two or multiple block structure on the problem and simply perform one cycle of Gauss-Seidel update to the block variables or update one randomly selected block. Global sublinear convergence of these methods has also been established. Exhausting all such works is impossible and out of scope of this paper. We refer interested readers to \cite{boyd2011distributed, he2012-rate-drs, glowinski2014alternating, ouyang2015accelerated, deng2016global, GXZ-RPDCU2016, xu2017accelerated-pdc, xu2017async-pd} and the references therein.

\subsection*{General convex problems}
As there are nonlinear inequality constraints, we do not find any work in the literature showing the global convergence rate of iALM, though its local convergence rate has been extensively studied (e.g., \cite{rockafellar1973dual, bertsekas1973convergence, rockafellar1976augmented}). Many existing works on nonlinearly constrained convex problems employ Lagrangian function instead of the augmented one and establish global convergence rate through dual subgradient approach (e.g., \cite{nedic2009subgradient, nedic2009approximate, necoara2014rate}). For general convex problems, these methods enjoy $O(1/\sqrt{k})$ convergence, and for strongly convex case, the rate can be improved to $O(1/k)$. To achieve an $\vareps$-optimal solution, compared to our results, their iteration complexity is $O(\vareps^{-1})$ times worse for the convex problems and $O(\vareps^{-\frac{1}{2}})$ worse for the strongly convex problems. Assuming Lipschitz continuity of $\nabla f_i$ for every $i\in [m]$, \cite{yu2017simple} proposes a new primal-dual type algorithm for nonlinearly constrained convex programs. Every iteration, it minimizes a proximal Lagrangian function and updates the multiplier in a novel way. With sufficiently large proximal parameter that depends on the Lipschitz constants of $f_i$'s, the algorithm converges in $O(1/k)$ ergodic rate. The follow-up paper \cite{yu2016primal} focuses on smooth constrained convex problems and proposes a linearized variant of the algorithm in \cite{yu2017simple}. Assuming compactness of the set $\cX$, it also establishes $O(1/k)$ ergodic convergence of the linearized method. 

\subsection*{Iteration complexity from existing results on iPPA}
Through relating iALM and iPPA, iteration complexity result can be obtained from existing results about iPPA to produce near-optimal dual solution. On solving problem $\min_\vz \phi(\vz)$, \cite{guler1992new} analyzes the iPPA with iterative update:
$$\vz^{k+1}\approx\argmin_{\vz} \phi(\vz) + \frac{1}{2\beta_k}\|\vz-\hat{\vz}^k\|^2.$$
If the above approximation error satisfies
\begin{equation}\label{eq:approx-phi}
\|\vz^{k+1}-\prox_{\beta_k \phi}(\hat{\vz}^k)\|=O(1/k^a),
\end{equation}
for a certain number $a > \frac{1}{2}$,
and the parameter $\beta_k$ is increasing, then by choosing specifically designed $\hat{\vz}^k$, \cite{guler1992new} shows that 
$$\phi(\vz^k)-\phi(\vz^*)=O(1/k^2)+O(1/k^{2a-1}).$$
From our discussion in section \ref{sec:rel-ialm-ippa}, if $\vareps_k=O(\frac{1}{k^{2a}\beta_k})$ in \eqref{eq:ineq-ialm-x}, then we have \eqref{eq:approx-phi} holds with $\phi=-d_0$, and thus obtain the convergence rate in terms of dual function:
$$d_0(\vz^*)-d_0(\vz^k)=O(1/k^2)+O(1/k^{2a-1}).$$
Note that $\vz^k$ is bounded from the summability of $\beta_k\vareps_k$ and the proof of Lemma \ref{lem:bd-z-Lz}. Hence, setting $\beta_k$ to a constant for all $k$ and applying Nesterov's optimal first-order method to each subproblem in \eqref{eq:ineq-ialm-x}, we need $O(k^a)$ gradient evaluations.

Let $a=\frac{3}{2}$. Then $K=O(1/\sqrt{\vareps})$ iPPA iterations are required to obtain an $\vareps$-optimal dual solution, i.e., $d_0(\vz^K)\ge d_0(\vz^*)-\vareps$, and the total number of gradient evaluations is
$$T_K=\sum_{k=1}^K O(k^{\frac{3}{2}})=O(K^{\frac{5}{2}})=O(\vareps^{-\frac{5}{4}}).$$
However, it is not clear how to measure the quality of the primal iterates. 

\section{Numerical results}\label{sec:numerical}
In this section, we conduct numerical experiments on the quadratically constrained quadratic programming (QCQP):
\begin{equation}\label{eq:qcqp}
\begin{aligned}
\Min_{\vx\in\RR^n} &~ \frac{1}{2}\vx^\top \vQ_0\vx + \vc_0^\top\vx, \\
\st & ~\frac{1}{2}\vx^\top \vQ_j\vx + \vc_j^\top\vx + d_j \le 0, \, j=1,\ldots,m,\\
&~ x_i \in [l_i, u_i], \, i=1,\ldots, n.
\end{aligned}
\end{equation}
The purpose of the tests is to verify the established theoretical results and compare the iALM with three different settings of parameters. 

Three QCQP instances are made. The first two instances are convex, and the third one is strongly convex. For all three instances, we set $n=100, m=5$ and $l_i=-1, u_i=1,\,\forall i$. The vectors $\vc_j, j=0,1,\ldots,m$ are generated following Gaussian distribution, and the scalars $d_j,j=1,\ldots,m$ are made negative. This way, all inequalities in \eqref{eq:qcqp} hold strictly at the origin $\vx=\vzero$, and thus the KKT conditions are satisfied at the optimal solution. $\vQ_j, j=0,1,\ldots,m$ are randomly generated and symmetric positive semidefinite. $\vQ_0$ is rank-deficient for the first two instances and full-rank for the third one. The data in the first two instances are the same except $\vQ_0$, which is 100 times in the second instance as that in the first instance. 

For all tests, we set $\vareps=10^{-3}$, $C_\beta=1$, $C_\vareps=\|\vu-\vl\|$, and $K=10$, and the initial primal-dual point is set to zero vector. The algorithm parameters $\{(\beta_k,\rho_k,\vareps_k)\}_{k=0}^{K-1}$ are set in three different ways corresponding to Theorems \ref{thm:it-comp-const}, \ref{thm:iter-comp-geo-const-eps}, and \ref{thm:iter-comp-geo} respectively, where $\sigma=10$ is used for the geometrically increasing penalty. On finding $\vx^{k+1}$ by applying Algorithm \ref{alg:apg} to $\min_{\vx\in\cX} \cL_{\beta_k}(\vx,\vz^k)$, we terminate the algorithm if the iteration number exceeds $10^6$ or
\begin{equation}\label{eq:term-cond}
\dist\left(-\nabla_\vx \cL_{\beta_k}(\vx^{k+1},\vz^k), \cN_\cX(\vx^{k+1})\right) \le \frac{\vareps_k}{\|\vu-\vl\|},
\end{equation}
where $\cX=\times_{i=1}^n[l_i,u_i]$.
Since $\cL_{\beta_k}(\vx,\vz^k)$ is convex about $\vx$, and $\|\vu-\vl\|$ is the diameter of the feasible set $\cX$, the condition in \eqref{eq:term-cond} guarantees that $\vx^{k+1}$ satisfies \eqref{eq:ialm-x}.

We report the difference of objective value and optimal value, and the feasibility violation at both actual iterate $\vx^k$ and the weighted averaged point $\bar{\vx}^k=\sum_{t=1}^k\vx^t/\sum_{t=1}^k \beta_t$. The optimal solution is computed by CVX \cite{grant2008cvx}. In addition, to compare the iteration complexity, we also report the number of gradient evaluations and function evaluations for each outer iteration. The results are provided in Tables \ref{table:cvx-med-Lip}, \ref{table:cvx-big-Lip}, and \ref{table:scvx} respectively for the three instances. In Table \ref{table:cvx-med-Lip}, we also report the results from quadratic penalty method, which corresponds to setting $K=1$ (see the discussions in Remark \ref{remark:penalty-comp}). 

From the results, we can clearly see that the quadratic penalty method is worse, namely, running a single iALM step with a big penalty parameter is significantly worse than running multiple steps with smaller penalty parameters. Also, we see that the iALM with three different settings yields the last actual iterate $\vx^K$ and the averaged point $\bar{\vx}^K$ of similar accuracy. For all three instances, to produce similarly accurate solutions, the iALM with constant penalty requires more gradient and function evaluations than that with geometrically increasing penalty. Furthermore, the iALM with geometrically increasing penalty and constant error requires fewest gradient and function evaluations on the first and third instances. However, the setting of geometrically increasing penalty and adaptive error is the best for iALM on the second instance. That is because the gradient Lipschitz constant of the objective in the second instance is significantly bigger than that in the first instance, in which case the bound on $T_K$ in \eqref{eq:T-ialm-cvx} is smaller than that in \eqref{eq:T-ialm-cvx-const-eps}.

\begin{table}[h]\caption{Results by quadratic penalty method (i.e., iALM with $K=1$) and iALM with three different settings on solving an instance of the QCQP problem \eqref{eq:qcqp}. In this instance, $\vQ_j$ is symmetric positive semidefinite for each $j=0,1,\ldots,m$, and $\vQ_0$ is singular. All $\vQ_j$'s have similarly large spectral norm.}\label{table:cvx-med-Lip}
\begin{center}
\begin{tabular}{|c||c|c|c|c|c|c|}
\hline
\#OutIter &  \#gradEval &  \#funEval  & $|f_0(\vx^k)-f_0^*|$   &$\big\|[\vf(\vx^k)]_+\big\|$   & $|f_0(\bar{\vx}^k)-f_0^*|$ &  $\big\|[\vf(\bar{\vx}^k)]_+\big\|$\\[1pt]\hline\hline
\multicolumn{7}{|c|}{quadratic penalty method}\\\hline\hline
& 1000000 & 2709547 & 7.4625e-06 & 8.6425e-05  & 7.4625e-06 & 8.6425e-05\\ \hline\hline
\multicolumn{7}{|c|}{constant penalty and constant error}\\\hline\hline
0 &   &  & 1.9949e+01 & 0.0000e+00  & 1.9949e+01 & 0.0000e+00\\
1 & 116670   & 316150 & 7.4618e-05 & 8.6422e-04  & 7.4618e-05 & 8.6422e-04\\
2 & 25   & 69 & 6.8617e-08 & 1.3444e-08  & 3.7275e-05 & 4.3212e-04\\
3 & 1   & 2 & 7.0236e-08 & 2.0003e-09  & 2.4826e-05 & 2.8807e-04\\
4 & 18   & 50 & 6.9210e-08 & 6.3998e-09  & 1.8603e-05 & 2.1606e-04\\
5 & 1   & 2 & 6.9964e-08 & 5.1992e-09  & 1.4868e-05 & 1.7285e-04\\
6 & 15   & 42 & 6.9288e-08 & 5.5183e-09  & 1.2379e-05 & 1.4404e-04\\
7 & 1   & 2 & 6.9923e-08 & 1.9898e-09  & 1.0600e-05 & 1.2346e-04\\
8 & 1   & 2 & 6.8851e-08 & 1.1481e-08  & 9.2665e-06 & 1.0803e-04\\
9 & 11   & 31 & 6.9962e-08 & 3.5627e-09  & 8.2291e-06 & 9.6026e-05\\
10 & 15   & 42 & 6.9431e-08 & 1.6670e-09  & 7.3993e-06 & 8.6423e-05\\\hline\hline

\multicolumn{7}{|c|}{geometrically increasing penalty and constant error}\\\hline\hline
0 &   &  & 1.9949e+01 & 0.0000e+00  & 1.9949e+01 & 0.0000e+00\\
1 & 54   & 150 & 4.8272e+00 & 1.4965e+02  & 4.8272e+00 & 1.4965e+02\\
2 & 27   & 73 & 4.8244e+00 & 1.4619e+02  & 4.8249e+00 & 1.4650e+02\\
3 & 43   & 117 & 4.6109e+00 & 1.1371e+02  & 4.6448e+00 & 1.1675e+02\\
4 & 40   & 111 & 2.6482e+00 & 3.9785e+01  & 2.9364e+00 & 4.6129e+01\\
5 & 89   & 246 & 3.3958e-01 & 4.1057e+00  & 6.4687e-01 & 7.8390e+00\\
6 & 217   & 594 & 3.6558e-03 & 4.4944e-02  & 6.9536e-02 & 8.0931e-01\\
7 & 72   & 202 & 4.9778e-06 & 6.8024e-05  & 6.9757e-03 & 8.0847e-02\\
8 & 9   & 31 & 6.0090e-08 & 1.3942e-07  & 6.9770e-04 & 8.0834e-03\\
9 & 1   & 9 & 6.3775e-08 & 1.7150e-08  & 6.9714e-05 & 8.0833e-04\\
10 & 11   & 42 & 6.4092e-08 & 8.7057e-10  & 6.9137e-06 & 8.0832e-05\\\hline\hline

\multicolumn{7}{|c|}{geometrically increasing penalty and adaptive error}\\\hline\hline
0 &   &  & 1.9949e+01 & 0.0000e+00  & 1.9949e+01 & 0.0000e+00\\
1 & 9   & 28 & 4.7815e+00 & 1.4467e+02  & 4.7815e+00 & 1.4467e+02\\
2 & 1   & 2 & 4.8000e+00 & 1.4311e+02  & 4.7990e+00 & 1.4325e+02\\
3 & 6   & 16 & 4.6446e+00 & 1.1719e+02  & 4.6705e+00 & 1.1968e+02\\
4 & 12   & 33 & 2.6268e+00 & 3.9395e+01  & 2.9242e+00 & 4.5842e+01\\
5 & 25   & 72 & 3.4073e-01 & 4.1139e+00  & 6.4650e-01 & 7.8356e+00\\
6 & 61   & 172 & 3.4865e-03 & 4.3037e-02  & 6.9338e-02 & 8.0737e-01\\
7 & 70   & 198 & 3.5015e-06 & 1.2314e-04  & 6.9557e-03 & 8.0705e-02\\
8 & 251   & 687 & 1.2912e-06 & 3.2057e-06  & 6.9477e-04 & 8.0709e-03\\
9 & 861   & 2340 & 3.2891e-07 & 1.8655e-07  & 6.9220e-05 & 8.0697e-04\\
10 & 2798   & 7588 & 8.1672e-08 & 8.2282e-09  & 6.8563e-06 & 8.0640e-05\\\hline

\end{tabular}

\end{center}
\end{table}

\begin{table}\caption{Results by iALM with three different settings on solving an instance of the QCQP problem \eqref{eq:qcqp}. In this instance, $\vQ_j$ is symmetric positive semidefinite for each $j=0,1,\ldots,m$. $\vQ_0$ is singular, and its spectral norm is about 100 times of that of every other $\vQ_j$.}\label{table:cvx-big-Lip}
\begin{center}
\begin{tabular}{|c||c|c|c|c|c|c|}
\hline
\#OutIter &  \#gradEval &  \#funEval  & $|f_0(\vx^k)-f_0^*|$   &$\big\|[\vf(\vx^k)]_+]\big\|$   & $|f_0(\bar{\vx}^k)-f_0^*|$ &  $\big\|[\vf(\bar{\vx}^k)]_+\big\|$\\[1pt]\hline\hline

\multicolumn{7}{|c|}{constant penalty and constant error}\\\hline\hline
0 &   &  & 2.4292e+00 & 0.0000e+00  & 2.4292e+00 & 0.0000e+00\\
1 & 1000000   & 2709541 & 1.9905e-03 & 3.5756e-03  & 1.9905e-03 & 3.5756e-03\\
2 & 534315   & 1447734 & 2.5923e-05 & 0.0000e+00  & 9.4283e-04 & 0.0000e+00\\
3 & 96   & 260 & 1.2144e-08 & 2.0371e-08  & 6.1404e-04 & 0.0000e+00\\
4 & 4   & 11 & 1.8289e-08 & 1.8936e-09  & 4.5509e-04 & 0.0000e+00\\
5 & 1   & 2 & 1.8757e-08 & 9.2594e-09  & 3.6146e-04 & 0.0000e+00\\
6 & 1   & 2 & 1.8790e-08 & 5.1559e-13  & 2.9977e-04 & 0.0000e+00\\
7 & 1   & 2 & 1.8601e-08 & 2.7707e-09  & 2.5606e-04 & 0.0000e+00\\
8 & 1   & 2 & 1.8639e-08 & 1.3276e-10  & 2.2347e-04 & 0.0000e+00\\
9 & 1   & 2 & 1.8617e-08 & 1.0515e-09  & 1.9824e-04 & 0.0000e+00\\
10 & 1   & 2 & 1.8643e-08 & 2.5039e-11  & 1.7813e-04 & 0.0000e+00\\\hline\hline

\multicolumn{7}{|c|}{geometrically increasing penalty and constant error}\\\hline\hline
0 &   &  & 2.4292e+00 & 0.0000e+00  & 2.4292e+00 & 0.0000e+00\\
1 & 1608   & 4372 & 8.7897e+00 & 9.3634e+01  & 8.7897e+00 & 9.3634e+01\\
2 & 176   & 477 & 8.7896e+00 & 9.3580e+01  & 8.7896e+00 & 9.3585e+01\\
3 & 9862   & 26722 & 8.7023e+00 & 7.5071e+01  & 8.7111e+00 & 7.6245e+01\\
4 & 5273   & 14288 & 8.3403e+00 & 5.7871e+01  & 8.3897e+00 & 5.9057e+01\\
5 & 3607   & 9773 & 3.2122e+00 & 1.3501e+01  & 3.7378e+00 & 1.6465e+01\\
6 & 2015   & 5460 & 4.9389e-01 & 1.5196e+00  & 8.2693e-01 & 2.6282e+00\\
7 & 2280   & 6183 & 1.5237e-02 & 4.4585e-02  & 9.7553e-02 & 2.8390e-01\\
8 & 1711   & 4643 & 5.4333e-05 & 1.5928e-04  & 9.8224e-03 & 2.8279e-02\\
9 & 29   & 84 & 6.2151e-08 & 2.2139e-07  & 9.8249e-04 & 2.8256e-03\\
10 & 4   & 17 & 9.9229e-09 & 3.9442e-09  & 9.8242e-05 & 2.8254e-04\\\hline\hline

\multicolumn{7}{|c|}{geometrically increasing penalty and adaptive error}\\\hline\hline
0 &   &  & 2.4292e+00 & 0.0000e+00  & 2.4292e+00 & 0.0000e+00\\
1 & 189   & 528 & 8.5613e+00 & 8.2337e+01  & 8.5613e+00 & 8.2337e+01\\
2 & 6   & 16 & 8.5659e+00 & 8.1843e+01  & 8.5656e+00 & 8.1888e+01\\
3 & 228   & 617 & 8.6240e+00 & 6.8449e+01  & 8.6188e+00 & 6.8342e+01\\
4 & 1074   & 2910 & 8.3440e+00 & 5.7951e+01  & 8.3836e+00 & 5.8797e+01\\
5 & 1529   & 4143 & 3.2289e+00 & 1.3596e+01  & 3.7525e+00 & 1.6552e+01\\
6 & 869   & 2355 & 4.9425e-01 & 1.5203e+00  & 8.2875e-01 & 2.6347e+00\\
7 & 1666   & 4518 & 1.5018e-02 & 4.3948e-02  & 9.7515e-02 & 2.8380e-01\\
8 & 1375   & 3731 & 5.5686e-05 & 1.6697e-04  & 9.8198e-03 & 2.8274e-02\\
9 & 139   & 383 & 5.9937e-08 & 1.0896e-06  & 9.8211e-04 & 2.8248e-03\\
10 & 967   & 2627 & 4.8356e-08 & 7.3292e-09  & 9.8168e-05 & 2.8237e-04\\\hline

\end{tabular}

\end{center}
\end{table}

\begin{table}\caption{Results by iALM with three different settings on solving a strongly convex instance of the QCQP problem \eqref{eq:qcqp}.}\label{table:scvx}
\begin{center}
\begin{tabular}{|c||c|c|c|c|c|c|}
\hline
\#OutIter &  \#gradEval &  \#funEval  & $|f_0(\vx^k)-f_0^*|$   &$\big\|[\vf(\vx^k)]_+\big\|$   & $|f_0(\bar{\vx}^k)-f_0^*|$ &  $\big\|[\vf(\bar{\vx}^k)]_+\big\|$\\[1pt]\hline\hline

\multicolumn{7}{|c|}{constant penalty and constant error}\\\hline\hline

0 &   &  & 1.3704e+01 & 0.0000e+00  & 1.3704e+01 & 0.0000e+00\\
1 & 4111   & 11170 & 1.6951e-05 & 4.1227e-04  & 1.6951e-05 & 4.1227e-04\\
2 & 10   & 28 & 4.5144e-08 & 9.5417e-09  & 8.4530e-06 & 2.0614e-04\\
3 & 1   & 2 & 4.5496e-08 & 1.1679e-09  & 5.6202e-06 & 1.3742e-04\\
4 & 1   & 2 & 4.5470e-08 & 0.0000e+00  & 4.2038e-06 & 1.0307e-04\\
5 & 1   & 2 & 4.5410e-08 & 5.2874e-10  & 3.3539e-06 & 8.2455e-05\\
6 & 1   & 2 & 4.5423e-08 & 1.1857e-10  & 2.7874e-06 & 6.8712e-05\\
7 & 1   & 2 & 4.5417e-08 & 1.4938e-10  & 2.3827e-06 & 5.8896e-05\\
8 & 1   & 2 & 4.5420e-08 & 1.4935e-11  & 2.0792e-06 & 5.1534e-05\\
9 & 1   & 2 & 4.5417e-08 & 4.2566e-11  & 1.8431e-06 & 4.5808e-05\\
10 & 1   & 2 & 4.5417e-08 & 5.2216e-12  & 1.6543e-06 & 4.1227e-05\\\hline\hline

\multicolumn{7}{|c|}{geometrically increasing penalty and constant error}\\\hline\hline
0 &   &  & 1.3704e+01 & 0.0000e+00  & 1.3704e+01 & 0.0000e+00\\
1 & 16   & 47 & 7.7888e-01 & 4.0032e+01  & 7.7888e-01 & 4.0032e+01\\
2 & 6   & 18 & 7.7879e-01 & 3.9621e+01  & 7.7880e-01 & 3.9658e+01\\
3 & 9   & 26 & 7.7142e-01 & 3.5936e+01  & 7.7269e-01 & 3.6303e+01\\
4 & 11   & 31 & 5.6681e-01 & 1.8568e+01  & 6.0051e-01 & 2.0298e+01\\
5 & 26   & 75 & 8.2135e-02 & 2.0563e+00  & 1.4945e-01 & 3.8384e+00\\
6 & 56   & 158 & 1.1249e-03 & 2.7458e-02  & 1.6675e-02 & 4.0689e-01\\
7 & 31   & 91 & 1.5987e-06 & 4.0394e-05  & 1.6772e-03 & 4.0708e-02\\
8 & 2   & 12 & 4.9762e-08 & 0.0000e+00  & 1.6776e-04 & 4.0703e-03\\
9 & 2   & 11 & 3.6101e-08 & 3.0525e-08  & 1.6745e-05 & 4.0705e-04\\
10 & 2   & 10 & 3.6919e-08 & 1.0977e-09  & 1.6412e-06 & 4.0706e-05\\\hline\hline

\multicolumn{7}{|c|}{geometrically increasing penalty and adaptive error}\\\hline\hline
0 &   &  & 1.3704e+01 & 0.0000e+00  & 1.3704e+01 & 0.0000e+00\\
1 & 1   & 6 & 2.1524e+00 & 1.8206e+01  & 2.1524e+00 & 1.8206e+01\\
2 & 1   & 6 & 1.8807e-01 & 2.3036e+01  & 2.4599e-01 & 2.2315e+01\\
3 & 1   & 6 & 3.6907e-01 & 2.8225e+01  & 3.3200e-01 & 2.7560e+01\\
4 & 3   & 12 & 5.1655e-01 & 1.7712e+01  & 5.1962e-01 & 1.8616e+01\\
5 & 8   & 25 & 8.4506e-02 & 2.2023e+00  & 1.4307e-01 & 3.8034e+00\\
6 & 24   & 71 & 8.4618e-04 & 3.1440e-02  & 1.5993e-02 & 4.0623e-01\\
7 & 75   & 210 & 4.8570e-05 & 3.3014e-05  & 1.5905e-03 & 4.0504e-02\\
8 & 250   & 684 & 5.5023e-06 & 4.6695e-06  & 1.5741e-04 & 4.0367e-03\\
9 & 801   & 2178 & 5.7505e-07 & 1.2058e-07  & 1.5607e-05 & 4.0198e-04\\
10 & 2603   & 7062 & 5.8030e-08 & 0.0000e+00  & 1.5503e-06 & 4.0009e-05\\\hline

\end{tabular}

\end{center}
\end{table}

\section{Concluding remarks}\label{sec:conclusion}
We have established ergodic and also nonergodic convergence rate results of iALM for general constrained convex programs. Furthermore, we have shown that to reach an $\vareps$-optimal solution, it is sufficient to evaluate gradients of smooth part in the objective and the functions in the inequality constraints for $O(\vareps^{-1})$ times if the objective is convex and $O(\vareps^{-\frac{1}{2}}|\log\vareps|)$ times if the objective is strongly convex. For the convex case, the iteration complexity result is optimal, and for the strongly convex case, the result is the best in the literature and appears to be also optimal.


\bibliographystyle{abbrv}
\bibliography{optim}

\end{document}

%% file: macros.tex
\usepackage{mathtools}

\usepackage[normalem]{ulem} 

\newcommand{\bm}[1]{\boldsymbol{#1}}

\newcommand{\va}{{\mathbf{a}}}
\newcommand{\vb}{{\mathbf{b}}}
\newcommand{\vc}{{\mathbf{c}}}

\newcommand{\vf}{{\mathbf{f}}}

\newcommand{\vl}{{\mathbf{l}}}

\newcommand{\vr}{{\mathbf{r}}}
\newcommand{\vs}{{\mathbf{s}}}

\newcommand{\vu}{{\mathbf{u}}}
\newcommand{\vv}{{\mathbf{v}}}

\newcommand{\vx}{{\mathbf{x}}}
\newcommand{\vy}{{\mathbf{y}}}
\newcommand{\vz}{{\mathbf{z}}}

\newcommand{\vA}{{\mathbf{A}}}
\newcommand{\vB}{{\mathbf{B}}}

\newcommand{\vQ}{{\mathbf{Q}}}

\newcommand{\vell}{{\bm{\ell}}}
\newcommand{\veps}{{\bm{\vareps}}}

\newcommand{\cL}{{\mathcal{L}}}
\newcommand{\cM}{{\mathcal{M}}}
\newcommand{\cN}{{\mathcal{N}}}

\newcommand{\cX}{{\mathcal{X}}}

\newcommand{\vareps}{\varepsilon}

\newcommand{\RR}{\mathbb{R}} 
\newcommand{\vzero}{\mathbf{0}} 

\newcommand{\dist}{\mathrm{dist}}    
\newcommand{\prox}{{\mathbf{prox}}} 
\newcommand{\dom}{{\mathrm{dom}}} 

\newcommand{\st}{\mbox{ s.t. }}

\DeclareMathOperator*{\argmin}{arg\,min} 
\DeclareMathOperator*{\argmax}{arg\,max} 

\DeclareMathOperator*{\Min}{minimize}
\DeclareMathOperator*{\Max}{maximize}

\newcommand{\bc}{\begin{center}}
\newcommand{\ec}{\end{center}}

\newcommand{\bdm}{\begin{displaymath}}
\newcommand{\edm}{\end{displaymath}}

\newcommand{\beq}{\begin{equation}}
\newcommand{\eeq}{\end{equation}}

\newcommand{\bfl}{\begin{flushleft}}
\newcommand{\efl}{\end{flushleft}}

\newcommand{\bt}{\begin{tabbing}}
\newcommand{\et}{\end{tabbing}}

\newcommand{\beqn}{\begin{eqnarray}}
\newcommand{\eeqn}{\end{eqnarray}}

\newcommand{\beqs}{\begin{align*}} 
\newcommand{\eeqs}{\end{align*}}  

\newtheorem{assumption}{Assumption}
\newtheorem{setting}{Setting}

%% file: ALMforNLP.bbl
\begin{thebibliography}{10}

\bibitem{FISTA2009}
A.~Beck and M.~Teboulle.
\newblock A fast iterative shrinkage-thresholding algorithm for linear inverse
  problems.
\newblock {\em SIAM journal on imaging sciences}, 2(1):183--202, 2009.

\bibitem{ben1997penalty-barrier}
A.~Ben-Tal and M.~Zibulevsky.
\newblock Penalty/barrier multiplier methods for convex programming problems.
\newblock {\em SIAM Journal on Optimization}, 7(2):347--366, 1997.

\bibitem{bertsekas1973convergence}
D.~P. Bertsekas.
\newblock Convergence rate of penalty and multiplier methods.
\newblock In {\em Decision and Control including the 12th Symposium on Adaptive
  Processes, 1973 IEEE Conference on}, volume~12, pages 260--264. IEEE, 1973.

\bibitem{bertsekas1999nonlinear}
D.~P. Bertsekas.
\newblock {\em Nonlinear programming}.
\newblock Athena scientific Belmont, 1999.

\bibitem{bertsekas2014constrained}
D.~P. Bertsekas.
\newblock {\em Constrained optimization and Lagrange multiplier methods}.
\newblock Academic press, 2014.

\bibitem{birgin2005numerical-alm}
E.~G. Birgin, R.~Castillo, and J.~M. Mart{\'\i}nez.
\newblock Numerical comparison of augmented lagrangian algorithms for nonconvex
  problems.
\newblock {\em Computational Optimization and Applications}, 31(1):31--55,
  2005.

\bibitem{boyd2011distributed}
S.~Boyd, N.~Parikh, E.~Chu, B.~Peleato, and J.~Eckstein.
\newblock Distributed optimization and statistical learning via the alternating
  direction method of multipliers.
\newblock {\em Foundations and Trends{\textregistered} in Machine Learning},
  3(1):1--122, 2011.

\bibitem{deng2016global}
W.~Deng and W.~Yin.
\newblock On the global and linear convergence of the generalized alternating
  direction method of multipliers.
\newblock {\em Journal of Scientific Computing}, 66(3):889--916, 2016.

\bibitem{GXZ-RPDCU2016}
X.~Gao, Y.~Xu, and S.~Zhang.
\newblock Randomized primal-dual proximal block coordinate updates.
\newblock {\em arXiv preprint arXiv:1605.05969}, 2016.

\bibitem{glowinski2014alternating}
R.~Glowinski.
\newblock On alternating direction methods of multipliers: a historical
  perspective.
\newblock In {\em Modeling, simulation and optimization for science and
  technology}, pages 59--82. Springer, 2014.

\bibitem{grant2008cvx}
M.~Grant, S.~Boyd, and Y.~Ye.
\newblock {CVX}: Matlab software for disciplined convex programming, 2008.

\bibitem{guler1991convergence}
O.~G{\"u}ler.
\newblock On the convergence of the proximal point algorithm for convex
  minimization.
\newblock {\em SIAM Journal on Control and Optimization}, 29(2):403--419, 1991.

\bibitem{guler1992new}
O.~G{\"u}ler.
\newblock New proximal point algorithms for convex minimization.
\newblock {\em SIAM Journal on Optimization}, 2(4):649--664, 1992.

\bibitem{he2010aalm}
B.~He and X.~Yuan.
\newblock On the acceleration of augmented lagrangian method for linearly
  constrained optimization.
\newblock {\em Optimization online}, 2010.

\bibitem{he2012-rate-drs}
B.~He and X.~Yuan.
\newblock On the ${O}(1/n)$ convergence rate of the douglas--rachford
  alternating direction method.
\newblock {\em SIAM Journal on Numerical Analysis}, 50(2):700--709, 2012.

\bibitem{hestenes1969multiplier}
M.~R. Hestenes.
\newblock Multiplier and gradient methods.
\newblock {\em Journal of optimization theory and applications}, 4(5):303--320,
  1969.

\bibitem{kang2015inexact}
M.~Kang, M.~Kang, and M.~Jung.
\newblock Inexact accelerated augmented lagrangian methods.
\newblock {\em Computational Optimization and Applications}, 62(2):373--404,
  2015.

\bibitem{kang2013accelerated}
M.~Kang, S.~Yun, H.~Woo, and M.~Kang.
\newblock Accelerated bregman method for linearly constrained
  $\ell_1$--$\ell_2$ minimization.
\newblock {\em Journal of Scientific Computing}, 56(3):515--534, 2013.

\bibitem{lan2013iteration-penalty}
G.~Lan and R.~D. Monteiro.
\newblock Iteration-complexity of first-order penalty methods for convex
  programming.
\newblock {\em Mathematical Programming}, 138(1-2):115--139, 2013.

\bibitem{lan2016iteration-alm}
G.~Lan, D.~Renato, and C.~Monteiro.
\newblock Iteration-complexity of first-order augmented lagrangian methods for
  convex programming.
\newblock {\em Mathematical Programming}, 155(1-2):511--547, 2016.

\bibitem{liu2016iALM}
Y.-F. Liu, X.~Liu, and S.~Ma.
\newblock On the non-ergodic convergence rate of an inexact augmented
  lagrangian framework for composite convex programming.
\newblock {\em arXiv preprint arXiv:1603.05738}, 2016.

\bibitem{necoara2014rate}
I.~Necoara and V.~Nedelcu.
\newblock Rate analysis of inexact dual first-order methods application to dual
  decomposition.
\newblock {\em IEEE Transactions on Automatic Control}, 59(5):1232--1243, 2014.

\bibitem{nedelcu2014computational}
V.~Nedelcu, I.~Necoara, and Q.~Tran-Dinh.
\newblock Computational complexity of inexact gradient augmented lagrangian
  methods: application to constrained mpc.
\newblock {\em SIAM Journal on Control and Optimization}, 52(5):3109--3134,
  2014.

\bibitem{nedic2009approximate}
A.~Nedi{\'c} and A.~Ozdaglar.
\newblock Approximate primal solutions and rate analysis for dual subgradient
  methods.
\newblock {\em SIAM Journal on Optimization}, 19(4):1757--1780, 2009.

\bibitem{nedic2009subgradient}
A.~Nedi{\'c} and A.~Ozdaglar.
\newblock Subgradient methods for saddle-point problems.
\newblock {\em Journal of optimization theory and applications},
  142(1):205--228, 2009.

\bibitem{nesterov2004introductory}
Y.~Nesterov.
\newblock {\em Introductory lectures on convex optimization: A basic course}.
\newblock Kluwer Academic Publisher, 2004.

\bibitem{nesterov2013gradient}
Y.~Nesterov.
\newblock Gradient methods for minimizing composite functions.
\newblock {\em Mathematical Programming}, 140(1):125--161, 2013.

\bibitem{ouyang2015accelerated}
Y.~Ouyang, Y.~Chen, G.~Lan, and E.~Pasiliao~Jr.
\newblock An accelerated linearized alternating direction method of
  multipliers.
\newblock {\em SIAM Journal on Imaging Sciences}, 8(1):644--681, 2015.

\bibitem{powell1969method}
M.~J. Powell.
\newblock {\em A method for non-linear constraints in minimization problems}.
\newblock in Optimization, R. Fletcher Ed., Academic Press, New York, NY, 1969.

\bibitem{rockafellar1973dual}
R.~T. Rockafellar.
\newblock A dual approach to solving nonlinear programming problems by
  unconstrained optimization.
\newblock {\em Mathematical programming}, 5(1):354--373, 1973.

\bibitem{rockafellar1973multiplier}
R.~T. Rockafellar.
\newblock The multiplier method of hestenes and powell applied to convex
  programming.
\newblock {\em Journal of Optimization Theory and applications},
  12(6):555--562, 1973.

\bibitem{rockafellar1976augmented}
R.~T. Rockafellar.
\newblock Augmented lagrangians and applications of the proximal point
  algorithm in convex programming.
\newblock {\em Mathematics of operations research}, 1(2):97--116, 1976.

\bibitem{schmidt2011convergence-ipg}
M.~Schmidt, N.~L. Roux, and F.~R. Bach.
\newblock Convergence rates of inexact proximal-gradient methods for convex
  optimization.
\newblock In {\em Advances in neural information processing systems}, pages
  1458--1466, 2011.

\bibitem{tseng1993convergence-exp-mult}
P.~Tseng and D.~P. Bertsekas.
\newblock On the convergence of the exponential multiplier method for convex
  programming.
\newblock {\em Mathematical Programming}, 60(1):1--19, 1993.

\bibitem{xu2017accelerated}
Y.~Xu.
\newblock Accelerated first-order primal-dual proximal methods for linearly
  constrained composite convex programming.
\newblock {\em SIAM Journal on Optimization}, 27(3):1459--1484, 2017.

\bibitem{xu2017async-pd}
Y.~Xu.
\newblock Asynchronous parallel primal-dual block update methods.
\newblock {\em arXiv preprint arXiv:1705.06391}, 2017.

\bibitem{xu2017accelerated-pdc}
Y.~Xu and S.~Zhang.
\newblock Accelerated primal--dual proximal block coordinate updating methods
  for constrained convex optimization.
\newblock {\em Computational Optimization and Applications}, pages 1--38, 2017.

\bibitem{yu2016primal}
H.~Yu and M.~J. Neely.
\newblock A primal-dual type algorithm with the ${O} (1/t)$ convergence rate
  for large scale constrained convex programs.
\newblock In {\em Decision and Control (CDC), 2016 IEEE 55th Conference on},
  pages 1900--1905. IEEE, 2016.

\bibitem{yu2017simple}
H.~Yu and M.~J. Neely.
\newblock A simple parallel algorithm with an ${O}(1/t)$ convergence rate for
  general convex programs.
\newblock {\em SIAM Journal on Optimization}, 27(2):759--783, 2017.

\end{thebibliography}
